\documentclass[onenumber, plainnumbering,english,11pt]{amsart}
\usepackage[square,numbers]{natbib}
\usepackage{amssymb,amsmath,amscd,mathrsfs,paralist}
\usepackage[mathscr]{eucal}
\usepackage[utf8]{inputenc}
\usepackage[T1]{fontenc}
\usepackage{bm}
\usepackage{graphicx}
\usepackage{tikz-cd}
\usepackage[all,cmtip]{xy}
\usepackage[colorlinks,linkcolor=black,urlcolor=black]{hyperref}
\hypersetup{
     colorlinks   = true,
     citecolor    = black}
\usepackage{pgfplots}
\newcommand{\R}{\mathbb R}
\newcommand{\N}{\mathbb N}

\newcommand{\C}{\mathbb C}

\newcommand{\lesss}{\rotatebox[origin=c]{90}{$\land$}}
\newcommand{\less}{\ \lesss\ }
\newcommand{\biggg}{\rotatebox[origin=c]{90}{$\lor$}}
\newcommand{\bg}{\ \biggg\ }
\newcommand{\nc}{\newcommand}
\nc{\BCc}{{\mathbb{C}(\wp(z),\wp^\prime(z))}}
\nc{\BC}{{\mathbb C}}
\nc{\BQ}{{\mathbb Q}}
\nc{\BR}{{\mathbb R}}
\nc{\BZ}{{\mathbb Z}}
\nc{\BP}{{\mathbb P}}
\nc{\BN}{{\mathbb N}}
\nc{\BM}{{\mathbb M}}
\nc{\fH}{{\mathfrak{H}}}
\nc{\vp}{{\varepsilon}}\nc{\dpar}{{\partial}}\nc{\al}{{\alpha}}
\nc{\PSL}{PSL(2,\BR)}
\nc{\PS}{PSL(2,\BZ)}
 \nc{\CL}{PSL(2,\BZ/m\BZ)}
 \newtheorem{theorem}{Theorem}[section]
\newtheorem{corollary}[theorem]{Corollary}
\newtheorem{lemma}[theorem]{Lemma}

\theoremstyle{definition}
\newtheorem{definition}[theorem]{Definition}
\newtheorem{remark}[theorem]{Remark}
\newtheorem{example}[theorem]{Example}

\newcommand{\theoref}[1]{Theorem~\ref{#1}}
\newcommand{\propref}[1]{Proposition~\ref{#1}}

\newcommand{\lemref}[1]{Lemma~\ref{#1}}

\newcommand{\corref}[1]{Corollary~\ref{#1}}

\newcommand{\secref}[1]{Section~\ref{#1}}

\numberwithin{equation}{section}  
\begin{document}
\title{Geometric methods in partial differential equations}
\author{Ahmed Sebbar}

\address{Chapman University\\
One University drive\
Orange, CA 92866}

\email{sebbar@chapman.edu,}
\author{Daniele Struppa}

\address{Chapman University\\
One University drive\
Orange, CA 92866}

\email{struppa@chapman.edu,}
 \author{Oumar Wone}
\address{Chapman University 
  \\One University drive, 92866, Orange, CA}
     \email{wone@chapman.edu}

 \keywords{Laplace's equation, Cayley-Bacharach theorem, Cohomology of sheaves}
 \subjclass[2010]{14N05, 14H05, 35C15}
  \thanks{The third author was partly supported by Chapman University. He expresses his warm thanks
to D. C. Struppa for the invitation and the excellent intellectual atmosphere during his stay.}
\date{}
\begin{abstract}
We study the interplay between geometry and partial differential equations. We show how the fundamental ideas we use require the ability to correctly calculate the dimensions of spaces associated to the varieties of zeros of the symbols of those differential equations. This brings to the center of the analysis several classical results from algebraic geometry, including the Cayley-Bacharach theorem and some of its variants as Serret's theorem, and the Brill-Noether Restsatz theorem.\end{abstract}
\maketitle 
\tableofcontents
\section{Introduction}
Many linear differential equations have their solutions represented by integrals. All the classical orthogonal polynomial verify a three terms linear relation, and according to a theorem of Favard admit an integral representation by measures for which they orthogonal. Many partial differential differential such as the Laplace equation, the heat equation or the wave equation have also their solutions represented by integrals. This can be considered as particular situation or as a concrete application of the fundamental principle of Ehrenpreis, that every solution of a system (in general, overdetermined) of homogeneous partial differential equations with constant coefficients can be represented as the integral with respect to an appropriate measure over the characteristic variety of the system. The later is an algebraic set, related to the symbol of the system. The main goal of this paper is We find that the theory was already present in its infancy in Bateman \cite{bateman1904}, where he showed how to give an integral representation for the solutions of the Laplace equation. Since the argument is essentially a counting argument. It is then not surprising that in order to obtain solution of suitable differential equations, we are forced to delve into existence theorems of the classical theory of algebraic curves, the Riemann-Roch theorem, the Cayley-Bacharach theorem, Serret’s theorem, and more generally those theorems that give us the dimensions of suitable cohomology groups, such as the Brill-Noether Restsatz theorem.

The plan of the paper is as follows: \secref{hitchin} is devoted to the use of differential geometry in partial differential equations, where we interpret following \cite{hitchin1982} the Laplace and ultrahyperbolic equations in terms of vector bundles. In \secref{bateman1} we study integral representations of general partial differential equations in three variables, whose symbols are homogeneous polynomials and we recover the results of \cite{hitchin1982, john1938}. We finally generalize the theory of \secref{bateman} to more than three variables in \secref{batemannew}. The paper concludes with one appendice of abstract material that is necessary to fully formalize our approach, but that we have isolated in an appendice to facilitate the reading of the rest of the article.
 
 The novelty of our presentation lies in the fact that throughout the paper we will always emphasize the discussion from the point of view of partial differential equations. We should conclude this introduction to point out that our own interest in this topic stems from our desire to better understand the interconnections between the theory of twistors and the fundamental principle of Ehrenpreis-Palamodov-Malgrange, at least in the case of elliptic and strictly hyperbolic systems. This interest was stimulated by our reading of Ehrenpreis’ \cite{leon1}, \cite{leon2} and we are currently investigating further how to clarify those interconnections.

\section{Differential geometry and partial differential equations}
\label{hitchin}
This part was inspired to us by \cite{hitchin1982}. It has been known for some time that problems in real differential geometry can often be simplified by using complex coordinates. For example in the plane $\R^2$ we can write $z=x+iy$ and thereby identify $\R^2\simeq\C$. We then discover that a $C^2$ function $f:\R^2\to\R$ is harmonic if and only if we can write it as
$$f=\psi+\overline{\psi}$$
where $\psi:\C\to\C$ is a holomorphic (complex-analytic) function. This is because a $C^2$ real-valued function $f$ is harmonic if and only if its laplacian is zero and the Laplace operator $\dfrac{\partial^2}{\partial x^2}+\dfrac{\partial^2}{\partial y^2}$ is proportional to the operator $\dfrac{\partial^2}{\partial z\partial \overline{z}}$. This shows us how to connect harmonic real-valued functions-an object from real differential geometry in the plane- to holomorphic functions of one complex variable- the natural objects  of complex analysis. If we try the same technique in $\R^3$ we have to accept the fact that odd dimensional spaces cannot be identified with complex spaces $\C^n$, for any integer $n$. We can however form another  space closely associated to the geometry of $T=\R^3$ that is intrinsically complex, and this is the fundamental idea behind twistor theory. Consider therefore the space $M$ of all oriented lines of $\R^3$. The generic element of the space $M$ is the oriented line $L(u,v)$ given by
$$L(u,v)=\{v+tu,t\in\R\}$$
where $||u||=1$ and $u$, $v\in\R^3$. Consider now the tangent bundle of the $2$-sphere $S^2$ defined by
$$TS^2=\{(u,v)\in\R^3\times\R^3:\|u\|=1,(u,v)=0\}$$
with $(u,v)$ denoting the Euclidean scalar product of $u$ and $v$. We can now define a bijection 
\begin{equation*}
\begin{split}
M&\to TS^2\\
L(u,v)&\mapsto(u,v-(v,u)u)
\end{split}
\end{equation*}
where the second component is the point on $L(u,v)$ closest to the origin of $T=R^3$. Remark that the map is indeed $TS^2$-valued and is clearly surjective. It is injective because if $(u,v-(v,u)u)=(u_1,v_1-(v_1,u_1)u_1)$ then $u=u_1$ and $v-v_1=(v-v_1,u)u$ which gives $L(u,v)=L(u_1,v_1)$. The mapping and its inverse mapping $(u,v)\in TS^2\mapsto L(u,v)\in M$ are smooth, a fact that shows that $M$ and $S^2$ are at least diffeomorphic. 

To get to the next stage, we recall that  the unit sphere $S^2$ can be endowed with a structure of complex manifold by choosing a covering atlas $\{U_0,U_1\}$, where $U_0=S^2\setminus\{(0,0,1)\}$ and $U_1=S^2\setminus\{(0,0,-1)\}$. We define complex coordinates on $U_0$ by
$$\xi_0(x,y,z)=\dfrac{x+iy}{1-z}$$
 which is the stereographic projection of the point $(x,y,z)$ from the north pole and on $U_1$ by
$$\xi_1(x,y,z)=\dfrac{x-iy}{1+z}$$
  which is the stereographic projection of the point $(x,y,z)$ from the south pole. We have by construction
  $$\xi_0(x,y,z)=\dfrac{1}{\xi_1(x,y,z)}=F(\xi_1(x,y,z))$$
  where $F(w)=\dfrac{1}{w}$, on $U_0\cap U_1$. This defines a complex structure on $S^2$. To define a complex structure on $TS^2$ we use standard constructions in differential geometry; for example a chart on $TS^2$ corresponding to the chart $\xi_0$ is given in local coordinates $(u,v)$ where $u\in U_0$ and $v\in \R^3$ by
 \begin{equation}
 \label{coor} 
 (\xi(u,v),\,\eta(u,v))=\left(\dfrac{u_1+iu_2}{1-u_3},\,d\xi_u(v)=\dfrac{v_1+iv_2}{1-u_3}+\dfrac{(u_1+iu_2)v_3}{(1-u_3)^2}\right).
 \end{equation}
  By definition the points of $M$ are oriented lines on $\R^3$. Moreover any point $p$ in $\R^3$ defines a $2$-sphere of lines, namely all oriented lines going through that point. Specifically the set of all lines through $p$ is the set of all $(u,v)\in S^2\times \R^3$ satisfying
  $$v=p-(p,u)u.$$
  We call this a {\it real section} of $M$ and denote it by $X_p$. Let us explore in more detail the geometry of these real sections.
  
  First we observe that these $X_p$ are called sections because the map
  \begin{equation}
  \begin{split}
  \rho_p:&S^2\to M\\
  &u\mapsto(u,p-(p,u)u)
  \end{split}
  \end{equation}
  defines a section of the projection $\pi:M\to S^2$ (namely $\pi(\rho_p(u))=u$, for all $u\in S^2$), and, with some abuse of notation, the image of this section is $X_p$. To understand why we called the sections $X_p$ {\it real} sections, we need to define a real structure on $M$, through a map
  $$\tau:M\to M$$
called a real structure. This map is defined as the involution that sends an oriented line to the same line with opposite orientation, i.e.
  $$\tau(u,v)=(-u,v).$$
  This real structure fixes the set $X_p$ because
  $$\tau(u,p-(p,u)u)=(-u,p-(p,u)u)=(-u,p-(p,-u)(-u)),$$ 
  and this explains why  $X_p$ is called a real section. 
  
  If $p=(x,y,z)$ is a point of $\R^3$ then $X_p$ is the set of all $(u,v)$ that correspond to lines through $p$:
  $$X_p=\{(u,p-(p,u)u),u\in S^2\}.$$
  If we substitute $v=p-(p,u)u$ into equation \eqref{coor} and simplify, we see that the equation of $X_p$ as a subset of $M$ is
 \begin{equation}
 \label{coor1} 
 \eta=\dfrac{1}{2}((x+iy)+2z\xi-(x-iy)\xi^2)
 \end{equation}
  when we identify $X_p$ with its image by the local chart given in equation \eqref{coor}. Hence under a similar identification, in coordinates, $\rho_p$ is given by
  $$\rho_p(\xi)=(\xi,\dfrac{1}{2}((x+iy)+2z\xi-(x-iy)\xi^2)).$$ 
   We will call any section that can be written in this way a holomorphic section. It is then possible to show that all holomorphic sections $S^2\to M$ take the form
   $$\xi\mapsto(\xi,a+b\xi+c\xi^2), a,b,c\in\C$$
   in local coordinates (this is because the holomorphic line $T\C P_1$ is the line bundle $\mathscr{O}(2)$ whose holomorphic sections are given by degree two homogeneous complex polynomials in two variables, hence by a second degree trinomial in non-homogeneous coordinates). With our choice of coordinates and the definition of a real structure one can show that if $(\xi,\eta)$ are the coordinates of a point $m\in M$, then $\left(\dfrac{-1}{\overline{\xi}},-\dfrac{\overline{\eta}}{\overline{\xi}^2}\right)$ are the coordinates of $\tau(m)$. So $\tau$ is anti-holomorphic. Therefore a section is real (i.e. invariant under the anti-holomorphic involution $\tau$) if and only if the equation 
   $$\eta=a+b\xi+c\xi^2$$
defines the same subset of $M$ as
$$-\dfrac{\overline{\eta}}{\overline{\xi}^2}=a+b\left(\dfrac{-1}{\overline{\xi}}\right)+c\dfrac{1}{\overline{\xi}^2}.$$  
  This immediately implies that $a=-\overline{c}$ and that $b$ is real. Hence the real sections defined by points of $\R^3$ as in  equation \eqref{coor1} are precisely all the real sections  of $M$. Thus we have a surjection between points of $\R^3$ and real sections of $M$. 
  The correspondence we have now established between $\R^3$ and $T$ is completely symmetric: points in $M$ define special subsets (oriented lines) in $\R^3$ and points in $\R^3$ define spacial subsets (holomorphic real sections) in $M$.
  
Set $\omega=g(\xi,\eta)d\xi$ a differential one form on $M$. If
$$\phi(x,y,z)=\int g(\xi,\dfrac{1}{2}((x+iy)+2z\xi-(x-iy)\xi^2))d\xi$$ 
  and we differentiate under the integral sign we have
  $$\dfrac{\partial^2\phi}{\partial x^2}+\dfrac{\partial^2\phi}{\partial y^2}+\dfrac{\partial^2\phi}{\partial x^2}=0$$
  that is $\phi$ is harmonic.

Consider now the ultrahyperbolic equation in $\R^4$ given by
$$\dfrac{\partial^2\varphi}{\partial x\partial y}-\dfrac{\partial^2\varphi}{\partial s\partial z}=0.$$
Let $T=\R^3$ and $f:\R^3\to\R$ be an arbitrary element of the Schwartz space $S(\R^3)$ and identify locally $M$ with $\R^4$. Choose local coordinates $(s,x,y,z)$ for $M$, on the open set where the third coordinate of $v$ does not vanish or again on the open set of lines which do not lie on the planes of constant $x_3$. A typical line in this open set is given by 
$$L=\{(s+ty,x+tz,t),t\in\R\}.$$
Define a function $\varphi$ on $M$ by 
$$\varphi(L)=\displaystyle\int_Lf$$
which gives in coordinates
$$\varphi(s,x,y,z)=\displaystyle\int_{-\infty}^{+\infty}f(s+ty,x+tz,t)dt.$$
Now there are four variables $s,x,y,z$ and $f$ is defined on $\R^3$ so we expect a differential condition on $\varphi$ (constraint). Indeed differentiating under the integral sign one has 
$$\dfrac{\partial^2\varphi}{\partial x\partial y}-\dfrac{\partial^2\varphi}{\partial s\partial z}=\displaystyle\int_{-\infty}^{+\infty}t(\dfrac{\partial^2}{\partial x\partial y}-\dfrac{\partial^2}{\partial x\partial y})f(s+ty,x+tz,t)dt=0.$$
It is natural to ask if this procedure, which goes under the name of John transform, can be inverted. This the case as shown by John in \cite{john1938}.

This example illustrates the defining philosophy of "twistor" theory. Namely, an unconstrained function on "twistor" space $T$ yields the solution to a differential equation on Minkowski space $M$, by means of an integral transform. We also have a simple geometric correspondence, another characteristic feature of twistor methods. More precisely we see
\begin{equation*}
\begin{split}
T&\longleftrightarrow M\\
\{\text{point in T}\}&\longrightarrow\{\text{oriented lines through point}\}\\
\{\text{line in T}\}&\longleftarrow \{\text{point in M}\}.
\end{split}
\end{equation*}
\begin{remark}
We recall that the tautological line bundle $H$ over $\C P_1$, also denoted by $\mathscr{O},$ is the holomorphic line bundle whose fibre over a point $\left[z\right]=\left[z_0:z_1\right]$ is given by the line $\left[z\right]$. This can be written also as $H=\{(\left[z\right],w)\,|\,w=\lambda z,\lambda\in \C-\{0\}\}\subset \C P_1\times \C^2$. The projection map $\pi:H\to\C P_1$ is the restriction of the projection from $\C P_1\times \C^2$, that is $(\left[z\right],w)\mapsto\left[z\right]$. By covering $\C P_1$ with the two open sets $U_0=\{\left[z\right]=\left[z_0:z_1\right],z_0\neq0\}$ and $U_1=\{\left[z\right]=\left[z_0:z_1\right],z_1\neq0\}$ we see that the transition function for the line bundle $H$ is given by
\begin{equation*}
\begin{split}
g_{01}&:U_0\cap U_1\to C^\times,\\
&\left[z\right]\mapsto\dfrac{z_1}{z_0}.
\end{split}
\end{equation*}
This follows from the fact that one can define local sections $\psi_i:U_i\to H,\,i\in\{0,1\}$ by 
$$\psi_0(\left[z\right])=(\left[z\right],(1,\dfrac{z_1}{z_0}))$$
and
$$\psi_1(\left[z\right])=(\left[z\right],(\dfrac{z_0}{z_1},1)).$$
And one sees that $\psi_0(\left[z\right])=\dfrac{z_1}{z_0}\psi_1(\left[z\right])$. By dualizing the line bundle $H$ one obtains the line bundle $\mathscr{O}(1)$ with transition function $g^\star_{01}(\left[z\right])=\dfrac{z_0}{z_1}$ and taking the tensor product of $\mathscr{O}(1)$ with itself one gets the line bundle $\mathscr{O}(2)$ with transition function given by $\left[z\right]\mapsto \left(\dfrac{z_0}{z_1}\right)^2$. 

Set $\xi=\dfrac{z_0}{z_1}$ on $U_1$ and $w=\dfrac{z_1}{z_0}$ on $U_0$, the two coordinates associated to $U_1$ and $U_0$ respectively. We have $\xi=\dfrac{1}{w}$ on $U_0\cap U_1$. This gives $d\xi=-\dfrac{1}{w^2}dw$, i.e. $dw=-\xi^{-2}d\xi$, and therefore $\partial_w=-\xi^2\partial_{\xi}$. This shows that the line bundles $\mathscr{O}(2)$ and $T\C P_1$ on $\C P_1$ have the same transition functions and as a consequence they are isomorphic.
\end{remark}
\section{The solution of partial differential equations by means of definite integrals}
\label{bateman1}

We want to begin this section with what Atiyah regarded as the beginning of twistor theory, intended as 
the representation of solutions of linear homogeneous (as a polynomial in the partial derivatives) partial differential equations with constant coefficients on $\R^n$ or $\C^n$ by means of definite integrals. We will soon need some results from classical algebraic geometry, but we begin here with a relatively simple example where all the calculations can be made explicit.
\subsection{The Laplace equation and some its solutions}
Consider two points  $P=(a,b,c)$ and $M=(x,y,z)$ in the usual Euclidean space $\R^3$, and assume they are subjected to Newtonian attraction, with $P$ being the attracting point, and $M$ the attracted one. By a suitable normalization we have that the force exerted by $P$ on $M$ is $\vec{F}=-\dfrac{\overrightarrow{PM}}{\|\overrightarrow{PM}\|^{3}}$. We set $r=\|\overrightarrow{PM}\|=\displaystyle\sqrt{(x-a)^2+(y-b)^2+(z-c)^2}$. Then the components of the force $\vec{F}$ are 
$$X=-\dfrac{x-a}{r^3},\quad Y=-\dfrac{y-b}{r^3}, \quad Z=-\dfrac{z-c}{r^3},$$
and this attraction derives from the potential
$$U(x,y,z)=\dfrac{1}{r},$$
since
$$\dfrac{\partial U}{\partial x}=\dfrac{\partial U}{\partial r}\dfrac{\partial r}{\partial x}=-\dfrac{x-a}{r^3}$$
(the same calculation holds for the other partial derivatives of $U$). 

If instead of an attracting point $P$, one has a finite attracting volume $V$, then the potential $U$ is given for points $(x,y,z)$ lying outside the volume $V$ by 
$$U(x,y,z)=\displaystyle\int\int\int_V\dfrac{dadbdc}{\sqrt{(x-a)^2+(y-b)^2+(z-c)^2}}.$$

By computing now the derivatives of $U(x,y,z)$ we obtain
$$\dfrac{\partial U}{\partial x}=-\displaystyle\int\int\int_V\dfrac{(x-a)dadbdc}{\left[(x-a)^2+(y-b)^2+(z-c)^2\right]^{\frac{3}{2}}}$$
and 
\begin{equation*}
\label{ }
\begin{split}
\dfrac{\partial^2U}{\partial x^2}&=-\displaystyle\int\int\int_V\dfrac{dadbdc}{\left[(x-a)^2+(y-b)^2+(z-c)^2\right]^{\frac{3}{2}}}\\
&+3\displaystyle\int\int\int_V\dfrac{(x-a)^2dadbdc}{\left[(x-a)^2+(y-b)^2+(z-c)^2\right]^{\frac{5}{2}}}.
\end{split}
\end{equation*}
This gives
$$\dfrac{\partial^2U}{\partial x^2}+\dfrac{\partial^2U}{\partial y^2}+\dfrac{\partial^2U}{\partial z^2}=0$$
and so
$$\Delta U=0.$$

In the introduction we made reference to how Whittaker, \cite{whittaker1927, whittaker1902}, found a way to write some solutions to this Laplace equation by means of definite integrals. We will now review in detail how that can be achieved.
 
The first observation is the fact that the Laplace differential operator $\Delta=\dfrac{\partial^2}{\partial x^2}+\dfrac{\partial^2}{\partial y^2}+\dfrac{\partial^2}{\partial z^2}$ is elliptic. This means that there is no non-zero element of $R^3$ which satisfies $x^2+y^2+z^2=0$.  Hence by the elliptic regularity theorem any solution to the Laplace partial differential equation is real analytic.

Let then $U(x,y,z)$ be a solution to the Laplace differential equation expressed as a convergent power series with respect to the three variables $x$, $y$, $z$, in the neighborhood of a given point $x_0$, $y_0$, $z_0$, and set
$$x=x_0+X,\quad y=y_0+Y,\quad z=z_0+Z.$$
The series 
$$U=a_0+a_1X+b_1Y+c_1Z+a_2X^2+b_2Y^2+c_2Z^2+2d_2YZ+2e_2ZX+2f_2XY+\ldots$$
is therefore convergent for $|X|+|Y|+|Z|$ sufficiently small. To determine the coefficients $a_0$, $a_1$, $\ldots$, we will calculate the second order partial derivatives of $U$ with respect to $X$, $Y$, $Z$, and put the resulting expressions in the equation
$$\dfrac{\partial^2U}{\partial X^2}+\dfrac{\partial^2U}{\partial Y^2}+\dfrac{\partial^2U}{\partial Z^2}=0.$$
By identification, this will give us linear relations from which we can deduce the values of the coefficients. Now note that if  we consider, in the series of $U$, the homogenous part $U_n$ of degree $n$ in $X$, $Y$, $Z$, the number of its coefficients is $\dfrac{(n+1)(n+2)}{2}$, because this is the dimension of the space of homogeneous polynomials of degree $n$ in three variables. As the laplacian is of second order, when $n\geqslant2$, its action on $U_n$ will give a homogeneous polynomial of degree $n-2$. This term has to vanish identically so we have $\dfrac{n(n-1)}{2}$ (dimension of the space of homogeneous polynomials of degree $n-2$ in three variables) linear relations among the coefficients of $U_n$. Therefore in the terms of degree $n$ of $U_n$, there will be 
$$\dfrac{(n+1)(n+2)}{2}-\dfrac{n(n-1)}{2}=2n+1$$
arbitrary coefficients when $n\geqslant2$ (the dimension of the space of homogeneous polynomials satisfying the Laplace equation is $2n+1$). Note that for $n=0,\,1$, $\dfrac{(n+1)(n+2)}{2}=2n+1$, and so there are $2n+1$ arbitrary coefficients in $U_n$ regardless of the value of $n.$ By superposition these terms will be linear combinations of $2n+1$ particular polynomial solutions, of degree $n$, to the Laplace equation. Let us look for such solutions (a basis of solutions).

For that let us start with the expression
$$E_n:=(Z+iX\cos(u)+iY\sin(u))^n,\,u\in\R,$$ 
which is clearly a solution to the Laplace equation of degree $n$. We can develop $E_n$ into a Fourier series because it is smooth and $2\pi$-periodic in $u$. This gives
$$\displaystyle \sum_{0}^\infty g_m(X,Y,Z)\cos(mu)+\sum_{0}^\infty h_j(X,Y,Z)\sin(ju);$$
with coefficients $g_m$ and $h_j$ linearly independent polynomials in $X,Y,Z$. 

However the development in Fourier series of $E_n$ contains only a finite number of terms. This follows by computing $E_n$ via the binomial formula, by linearizing the various powers of $\cos(u)$, $\sin(u)$ and by uniqueness of the Fourier expansion of a continuous $2\pi$-periodic function. Therefore one can write
$$E_n=\displaystyle \sum_{0}^n g_m(X,Y,Z)\cos(mu)+\sum_{1}^nh_j(X,Y,Z)\sin(ju)$$
where by Fourier one has
$$\pi g_m(X,Y,Z)=\displaystyle\int_{-\pi}^\pi(Z+iX\cos u+iY\sin u)^n\cos (mu) du$$
$$\pi h_j(X,Y,Z)=\displaystyle\int_{-\pi}^\pi(Z+iX\cos u+iY\sin u)^n\sin (ju) du.$$
To show that $E_n$ may be written in such a form one can use an induction based on the classical formulas, valid for  $a,b\in\R$
\begin{equation*}
\begin{split}
&\cos(a)\cos(b)=(\cos(a-b)+\cos(a+b))/2,\quad\sin(a)\sin(b)=(\cos(a-b)-\cos(a+b))/2\\
&\cos(a)\sin(b)=(\sin(a+b)-\sin(a-b))/2.
\end{split}
\end{equation*}
We remark that the $g_m$ are even in $Y$ and that the $h_j$ are odd in $Y$. For instance one has by definition
$$\pi g_m(X,-Y,Z)=\displaystyle\int_{-\pi}^\pi(Z+iX\cos u-iY\sin u)^n\cos (mu) du;$$
by setting $u=-v$, we obtain
$$\pi g_m(X,Y,Z)=\displaystyle\int_{\pi}^{-\pi}(Z+iX\cos v+iY\sin v)^n\cos (mv) (-dv)=\pi g_m(X,-Y,Z).$$
Also the highest power of $Z$ present in $g_m$ or $h_j$ is $n-m$ (respectively $n-j$). To see this one may use an induction based on the formula
 $$E_n=\displaystyle \sum_{0}^n g_m(X,Y,Z)\cos(mu)+\sum_{1}^nh_j(X,Y,Z)\sin(ju)$$
and the fact that 
$$E_{n+1}=(Z+iX\cos(u)+iY\sin(u))E_n.$$
Now we can use these properties of $g_m$ and $h_j$ to show that they are linearly independent (and therefore form a basis of the vector space of homogenous polynomials in $X,Y,Z$ solution to the Laplace equation, $X,Y,Z$ are still considered to be sufficiently small). Let $\lambda_0$, $\lambda_1$, $\ldots$, $\lambda_n$ and $\mu_1$, $\mu_2$, $\ldots$, $\mu_n$ be scalars such that
$$\lambda_0g_0+\ldots+\lambda_ng_n+\mu_1h_1+\ldots\mu_nh_n=0.$$
Then since the $g_m$ are even and the $h_j$ are odd, we have separately
$$\lambda_0g_0+\ldots+\lambda_ng_n=0$$
and
$$\mu_1h_1+\ldots\mu_nh_n=0.$$
Therefore from the fact that $g_m$ and $h_m$ are of degree $n-m$ in $Z$, one deduces immediately that all the coefficients $\lambda_m$ and $\mu_m$ are zero.
 
This being said every linear combination of the independent $2n+1$ solutions can then be put in the form
$$\displaystyle\int_{-\pi}^\pi(Z+iX\cos u+iY\sin u)^nf_n(u) du.$$
Here for each $n$ $f_n(u)=\frac{1}{\pi}(\sum_{0}^n\alpha_m\cos(mu)+\sum_{1}^n\beta_j\sin(ju))$, for some $\alpha_m$ and $\beta_j$. Assuming that $|X|+|Y|+|Z|<B<1$ (for instance) and choosing $D>0$ such that for each $n$ $|(\alpha_m)_{0\leqslant m\leqslant n}|<D$ and $|(\beta_j)_{1\leqslant j\leqslant n}|<D$ we shall have
$$U(X,Y,Z)=\sum_{0}^\infty\displaystyle\int_{-\pi}^\pi(Z+iX\cos u+iY\sin u)^nf_n(u) du\equiv\displaystyle\int_{-\pi}^\pi \Phi(Z+iX\cos u+iY, u) du,$$
for $\Phi$ a suitable well-defined function in two variables. We have therefore obtained a local integral representation of some solutions to the Laplace equation in $\R^3$.
\subsection{The kernel of the partial differential operator $F\left(\frac{\partial}{\partial x},\frac{\partial}{\partial y},\frac{\partial}{\partial z}\right)$}
\label{bateman}
In this subsection we will show how to generalize the result of the previous subsection to the case in which the Laplacian is replaced by another partial differential operator, whose symbol is still a homogeneous polynomial in three variables. As we will see, it is not so easy to calculate the dimensions of the spaces of coefficients in the series expansion of the solution, and therefore we have to resort to some pretty significant results from classical algebraic geometry.

To this end, we introduce some classical terminology before being able to prove the theorem stated later in this section. Let $C$ be a smooth projective plane algebraic curve and let $\tilde{C}$ be its associated compact Riemann surface. A divisor $D$ on $C$ is a formal sum $D=\sum_{p\in C} n_p p$ with $n_p$ an element of the set of integers and all but a finite number of the $n_p$'s are equal to zero. The degree of a divisor $D$ is defined by $\deg{D}:=\sum_{p\in C}n_p$. A divisor $D=\sum_{p\in C} n_p p$ is called effective, and we will write $D\geq 0$, when $n_p\geq 0$ for all $p\in C$. Two divisors $D$ and $D^\prime$ are linearly equivalent if there exists a rational function $f$ on $C$ such that $D-D^\prime=(f)$, where $(f)$ stands for $f^{-1}(0)-f^{-1}(\infty)$. Equivalence between divisors is easily seen to be an equivalence relation. One important example of a divisor class on a smooth algebraic curve is the  class of the divisor of any given rational differential on the curve. It will be called the canonical divisor class and will be denoted $K$. Let $\mathscr{K}\left(C\right)$ be the field of rational functions on $C$. Let us introduce now the following vector space associated to a given divisor $D$
\begin{equation}
\label{R1}
L(D)=\{0\}\cup\{f, f\in \mathscr{K}\left(C\right), (f)+D\geq0\}.
\end{equation}
The dimension of $L(D)$ is denoted by $l(D)$. Note that $l(D)$ only depends on the divisor class of $D$. The fundamental theorem which enables one to compute $l(D)$ in general is the Riemann-Roch theorem 
\begin{theorem}[R-R, \cite{griffiths1985}]
\label{RR}
Let $C$ be a smooth projective plane algebraic curve of degree $d$, $D$ a divisor on $C$, $K$ the canonical divisor class of $C$ and $g$ the genus of $C$. Then
$$l(D)=\deg{D}+1-g+l(K-D).$$
\end{theorem}
The following result is a simple consequence of the Riemann-Roch theorem. 
\begin{theorem}
Any rational function $f$ on a nonsingular curve plane $C$ can be written as the quotient of two homogeneous polynomials in three variables and of the same degree, restricted to $C$:
$$f=\displaystyle\dfrac{Q}{R}|_C$$ 
\end{theorem}
Let $L$ be a line on $\C P_2$ not containing $C$; we set $H:=\sum_{p\in L\cap C}I_p(C,L)p$, where $I_p(C,L)$ is the intersection multiplicity at $p$ between $C$ and $L$. More generally, let $X$ be any plane curve intersecting $C$ only in isolated points; then we define the divisor cut on $C$ by $X$, denoted by $C\cdot X$, by the formula $C\cdot X:=\sum_{p\in X\cap C}I_p(C,X)p$. Remark that any two such divisors $C\cdot X$ and $C\cdot X^\prime$ associated to different such curves $X$ and $X^\prime$ of the same degree, are linearly equivalent. This is because if $X=\{F=0\}$ and $X^\prime=\{F^\prime=0\}$ then $C\cdot X-C\cdot X^\prime=\left(\dfrac{F}{F^\prime}|_C\right)$.

 Let us now recall the Bézout's theorem
\begin{theorem}[Bézout]
If $C$ is a smooth plane curve of degree $d$. If $X$ is a plane curve of degree $e$ not containing $C$, then the degree of the divisor $C\cdot X$ cut by $X$ on $C$ is $d\cdot e$.
\end{theorem}
It follows from Bézout's theorem \cite[p.~86]{griffiths1985} that the degree of $H$ is equal to $n$.
We also have 
\begin{theorem}[L=AF+BG]\label{AFBG}
Let $C=\{F=0\}$ be a smooth curve and $X=\{G=0\}$ a curve not containing $C$. Then if a curve $Y=\{L=0\}$ contains $C\cdot X$  one can write $L=AF+BG$, with $A$ and $B$ homogeneous polynomials of degrees, respectively, $\deg(L)-\deg(F)$ and $\deg(L)-\deg(G)$.
\end{theorem}
One of the central results in the classical theory of plane algebraic curves is the following theorem of Brill and Noether
\begin{theorem}[Brill-Noether Restsatz]
\label{BN1}
Let $C$ be a non-singular plane curve, and let $X$ be any plane curve not containing $C$. Then for any divisor linearly equivalent to $C\cdot X$, there is a plane curve $X^\prime$ not containing $C$, and such that $C\cdot X^\prime=D$. 
\end{theorem}
An important important consequence of the Brill-Noether theorem \cite[cor.~6]{eisenbud1996} is 
\begin{corollary}
Let $C$ be a smooth plane curve of degree $d$, and let $\Lambda$ be a subset of $\lambda$ distinct points of $C$ considered as an effective divisor on $C$. Then 
the dimension of the space of homogeneous polynomials of degree $m$ in three variables modulo those vanishing on $\Lambda$ is equal to
\begin{equation}
\label{BN}
l(mH)-l(mH-\Lambda).
\end{equation}
\end{corollary}
\begin{proof}
We apply the Restsatz to the curve C and to the $m$-th power of a generic line. The vector space of homogeneous polynomials of degree $m$ cuts out on $C$ the family of divisors linearly equivalent to $mH$. This family denoted $|mH|$ is isomorphic to the projective space $P(L(mH))$. This set has the same dimension as the projective space associated to the vector space of homogeneous polynomials of degree m not containing $C$ or, in other words, the projective space of the vector space of homogeneous polynomials of degree $m$ modulo those homogeneous polynomials vanishing on $C$. 

To see this, take a divisor $D$ linearly equivalent to $mH$. We have $D=C\cdot X$ for some $X$ of degree $m$ not containing $C$. Therefore the map
\begin{equation*}
\begin{split}
P(V)&\to |mH|\\
X&\mapsto C\cdot X=D
\end{split}
\end{equation*}
is surjective. 
Suppose that $D=C\cdot X^\prime$ for another curve $X^\prime$. We show that the defining polynomials of $X$ and $X^\prime$ are proportional. The fact that $X$ and $X^\prime$ do not contain $C$ is equivalent to the fact that their defining polynomials $F$ and $F_1$ (resp) are not divisible by the polynomial $P$ which defines $C$. From Theorem \ref{AFBG} we can write $F=A_1F_1+B_1P$ and $F_1=A_2F+B_2P$ with $A_1$ and $A_2$ complex numbers. Moreover $F(1-A_1A_2)=(A_1B_2+B_1)P$ and $F_1(1-A_1A_2)=(A_2B_1+B_2)P$. So we have a contradiction unless $B_1=B_2=0$ and $F_1=\lambda_0 F$. Hence the map
\begin{equation*}
\begin{split}
P(V)&\to |mH|\\
X&\mapsto C\cdot X=D
\end{split}
\end{equation*}
is injective. 

We have thus shown that $P(V)\simeq|mH|\simeq P(L(mH))$. This gives $\dim V=l(mH)=\dim(L(mH))$.
Likewise one shows that $l(mH-\Lambda)=\dim L(mH-\Lambda)$ is the dimension of $V^\prime$, the vector space of homogeneous polynomials of degree m passing through the points of $\Lambda$ modulo those vanishing on $C$. This follows from the fact that $L(mH-\lambda)\simeq\Gamma(\tilde{C},\left[mH-\Lambda\right])$, where $\left[mH-\Lambda\right]$ is the line bundle, \cite[chap.~1]{griffiths1978}, associated to the divisor $mH-\Lambda$. This latter line bundle is $\mathscr{O}(m)|_{\widetilde{C}}\otimes\mathscr{O}_{\widetilde{C}}(\left[-\Lambda\right])$ and this concludes the proof.
\end{proof}
\begin{lemma}
Given a divisor $D$ on a nonsingular projective algebraic curve $C$, the set 
$$|D|:=\{D^\prime\sim D,D^\prime\geqslant0\}\simeq P(L(D)).$$
\end{lemma}
\begin{proof}
For every $D^\prime\in|D|$, there exists $f\in\mathscr{K}(C)$ such that $D^\prime=(f)+D$. And any two such $f\in\mathscr{K}(C)$ differ by a non-zero constant. Indeed  if $(f)+D=(g)+D$ then $(f)=(g)$ so $(f/g)=0$. Let us take two representatives of $f$ and $g$, still denoted $f$ and $g$ ($f\neq0$ and $g\neq0$). One sees then that $f/g$ has no zeros or poles (so it is in particular holomorphic on $\tilde{C}$) therefore it is an element of $\C^\times$ because $\mathscr{O}_C(C)=\C$. 
Therefore we have a bijective map
\begin{equation*}
\begin{split}
P(L(D))&\to|D|\\
f&\mapsto (f)+D
\end{split}
\end{equation*}
\end{proof}
We will need the following version of Cayley-Bacharach's theorem \cite[th.~CB4]{eisenbud1996}
\begin{theorem}[C-B1]
\label{CB1}
Let $X_1$, $X_2$ be plane curves of degrees $m$ and $n$ respectively, with $X_1$ smooth and meeting in a collection of $mn$ distinct points $\Gamma=\{p_1,\ldots,p_{mn}\}$. If $X\subset \C P_2$ is any plane curve of degree $m+n-3$ containing all but one point of $\Gamma$, then $X$ contains all of $\Gamma$.
\end{theorem}
\begin{corollary}[Chasles' theorem]
Let $X_1$, $X_2\subset \C P_2$ be cubic plane curves, with $X_1$ smooth, meeting in nine points $P_1,P_2,\ldots,P_9$. If $X\subset \C P_2$ is any cubic plane curve containing $8$ among them, then $X$ contains the remaining point as well. 
\end{corollary}
Let us introduce the following theorem of Serret \cite[p.~99]{semple1949}
\begin{theorem}[Serret]
\label{Serret}
Let $p=\left[a,b,c\right]$ be a point of $\C P_2$ and associate to it the linear form $l_p(x,y,z)=ax+by+cz$. Then the necessary and sufficient condition that a curve $C^r$, of degree $r$, which passes through $q-1$ of a set of $q$ given points of $\C P_2$, passes through the remaining one, is that there is a linear relation or syzygy (with all coefficients non-zero) connecting the $r^{th}$ powers of the linear forms (or, by abuse of language, tangential equation) associated to each given point.
\end{theorem}
\begin{proof}
If one has a relation of the form
$$\sum_{i=1}^q\lambda_i(l_{p_i})^r=0$$
and the equation of the curve $C^r$ is given by $h$, then $h\left(\dfrac{\partial}{\partial x},\dfrac{\partial}{\partial y},\dfrac{\partial}{\partial z}\right)(l_{p_i})^r=r!h(p_i)$. So if $q-1$ of the points lie on $C^r$ the remaining one also does. The converse is shown along similar lines, see \cite[p.~99]{semple1949}
\end{proof}
\begin{corollary}
\label{Serret1}
When $q=\dbinom{m+2}{m}$ Serret's theorem gives the necessary and sufficient condition that $q$ points should lie on a curve of degree $m$.
\end{corollary}
\begin{proof}
Indeed it follows from the Riemann-Roch theorem (see below in the proof of \theoref{CB2}) that there always exists a curve of degree $m$ which passes through any $q-1=\dfrac{1}{2}m(m+3)$ given points of the fixed smooth curve $C$. So if there is linear relation between the $m^{th}$ powers of the linear terms associated to the q points, then necessarily all the q points lie on that curve of degree $m$. The converse is the necessity statement of Serret's theorem.\end{proof}
We finally prove the following generalized Cayley-Bacharach theorem inspired by \cite{semple1949} and which easily follows from Serret's theorem and The Riemann-Roch theorem.

\begin{theorem}[C-B2]
\label{CB2}
Let $X_1$ and $X_2$ be plane curves of degrees $m$ and $n$ respectively, with $X_1$ smooth and meeting $X_2$ in a collection of $mn$ distinct points $\Gamma=\{p_1,\ldots,p_{mn}\}$. Every curve $C^{m+n-\gamma}$ $(\gamma\bg 3)$ of degree $m+n-\gamma$ which passes through $mn-\dfrac{1}{2}(\gamma-1)(\gamma-2)$ of the points $X_1\cap X_2$ passes through the remainder except when these remaining $\dfrac{1}{2}(\gamma-1)(\gamma-2)$ points lie on a curve $C^{\gamma-3}$ of degree $\gamma-3$.
\end{theorem}
\begin{proof}
Denote by $\left[a_s,b_s,c_s\right]$ the points of $X_1\cap X_2$. From the Cayley-Bacharach \theoref{CB1} every curve $C^{m+n-3}$ which passes through all but one point of $X_1\cap X_2$ necessarily passes through the remaining point. Therefore from Serret's theorem we have a syzygy
$$\sum_{s=1}^{mn}k_s(a_sx+b_sy+c_sz)^{m+n-3}=0, k_s\in \C^\times.$$
Since this last equation is an identity in $x,y,z$, it can be differentiated repeatedly with respect to the variables $x,y,z$. Then, if $F\left(\dfrac{\partial}{\partial x},\dfrac{\partial}{\partial y},\dfrac{\partial}{\partial z}\right)$ is a homogeneous polynomial of degree $\nu$ in the operators, we evidently have
$$\sum_{s=1}^{mn}k_sF(a_s,b_s,c_s)(a_sx+b_sy+c_sz)^{m+n-3-\nu}=0.$$
In particular, taking $\nu=m+n-\gamma$
$$\sum_{s=1}^{mn}k_sF(a_s,b_s,c_s)(a_sx+b_sy+c_sz)^{\gamma-3}=0.$$
Thus if $F$ is the equation of the curve $C^{m+n-\gamma}$ provided by the hypothesis of the theorem, we obtain an identity involving only $\dfrac{1}{2}(\gamma-1)(\gamma-2)$ of the points. But from the Riemann-Roch theorem we know that there always exists a curve $C^{\gamma-3}$ passing through any given $\dfrac{1}{2}\gamma(\gamma-3)$ given points of $X_1$. Indeed if $\Lambda$ is that set of points considered as an effective divisor
$$ l((\gamma-3)H)-l((\gamma-3)H-\Lambda)=\dfrac{1}{2}\gamma(\gamma-3)+l(K-(\gamma-3)H)-l(K-(\gamma-3)H+\Lambda),$$
with $K$ and $H$ defined with respect to $X_1$. Because $\Lambda\geq0$ it follows that $l(K-(\gamma-3)H)-l(K-(\gamma-3)H+\Lambda)\leq0$, by definition. Thus $l((\gamma-3)H)-l((\gamma-3)H-\Lambda)\leq\dfrac{1}{2}\gamma(\gamma-3)$ which is strictly less than the dimension $\dfrac{1}{2}(\gamma-1)(\gamma-2)$ of the vector space of homogeneous polynomials of degree $\gamma-3$ in three variables. Given the interpretation we gave of the quantity $l((\gamma-3)H)-l((\gamma-3)H-\Lambda)$, we conclude that there exists a non-trivial polynomial of degree $\gamma-3$ vanishing on $\Lambda$. Hence if the remaining $\dfrac{1}{2}(\gamma-1)(\gamma-2)$ points do not lie on a $C^{\gamma-3}$, they must all lie on the given curve of degree $m+n-\gamma$, by Serret's theorem. 
\end{proof}
We are finally ready for the main result of this work:

\begin{theorem} 
\label{second}
Let $C$ be a smooth projective plane curve of degree $n\geqslant2$ with equation given by $F(\xi,\eta,\zeta)=0$. Let  $\left[\xi,\eta,\zeta\right]$ be the coordinates of its points expressed as functions of a uniformizing parameter $t$. Then any real analytic solution $V$ of the equation \begin{equation}
\label{f3M}
\displaystyle F\left(\dfrac{\partial}{\partial x},\dfrac{\partial}{\partial y},\dfrac{\partial}{\partial z}\right) \phi(x,y,z)=0
\end{equation}
 on a sufficiently small open set can be put in the form
\begin{equation}
\label{f5}
\displaystyle V(x,y,z)=\int \Phi(\xi x+\eta y+\zeta z, t)dt,
\end{equation}
for a suitable path of integration.
\end{theorem}
\begin{proof}
Without loss of generality we assume $V$ to be real analytic near the origin as a function of $x$, $y$, $z$. Expand $V$ as an absolutely and uniformly convergent power series in $x$, $y$, $z$ near the origin itself.

We now apply $F\left(\dfrac{\partial}{\partial x},\dfrac{\partial}{\partial y},\dfrac{\partial}{\partial z}\right)$ to this series, and we set to zero the coefficients of the various powers. Let's look at the homogeneous part of $F\left(\dfrac{\partial}{\partial x},\dfrac{\partial}{\partial y},\dfrac{\partial}{\partial z}\right)V$ of degree $m$; when $m \geq n$ we find 
$$\dfrac{1}{2}(m-n+1)(m-n+2)$$
relations among the coefficients since the action of  $F\left(\dfrac{\partial}{\partial x},\dfrac{\partial}{\partial y},\dfrac{\partial}{\partial z}\right)$ on $V(x,y,z)$ will leave a polynomial of degree $m-n$ and $\dfrac{1}{2}(m-n+1)(m-n+2)$ represents the dimension of the space of homogeneous polynomials of degree $m-n$ in three variables. But when $m<n$ no such relations arise, because the operation of $F\left(\dfrac{\partial}{\partial x},\dfrac{\partial}{\partial y},\dfrac{\partial}{\partial z}\right)$ on the homogeneous parts of $V(x,y,z)$ of degree $m\less n$ will cancel them entirely.

As the dimension of the space of homogeneous polynomials of degree $m$ in three variables is $\dbinom{m+2}{m}$, the homogeneous part of degree $m$ of the solution $V$ is a linear combination of at most
$$\dfrac{1}{2}(m+1)(m+2)-\dfrac{1}{2}(m-n+1)(m-n+2)=mn+1-\dfrac{1}{2}(n-1)(n-2)$$
linearly independent terms when $m\geqslant n$, and of at most $\dfrac{1}{2}(m+1)(m+2)$ independent terms when $m<n$.

We first consider the case $m\geqslant n$. In order to express the terms of order $m$ into integral form \eqref{f5} we proceed as follows. If we take $M:=mn+1-\dfrac{1}{2}(n-1)(n-2)$ arbitrary points on the curve $F(\xi,\eta,\zeta)=0$ (belonging to the same domain $\mathscr{D}$ of the uniformizing parameter), then the corresponding powers (of linear forms) $(\xi x+\eta y+\zeta z)^m$ will in general be linearly independent.

For, if not, there would be a linear relation between them, of the form
\begin{equation}
\label{f6}
\displaystyle\sum_1^{M}\lambda_i(\xi_i x+\eta_i y+\zeta_i z)^m=0,\,\lambda_i\in\C,\,\forall \,i.
\end{equation}
Leaving out one of the $M$ points, say $(\xi_1,\eta_1,\zeta_1)$, we can draw a curve of the $m$-th order $f(\xi,\eta,\zeta)=0$ through the remaining $M-1$ points. Indeed with the notations introduced above we have
$$l(mH-\Lambda)=\deg(mH-\Lambda)+1-g+l(K-mH+\Lambda),$$
where $\Lambda$ is the set of the $mn-\dfrac{1}{2}(n-1)(n-2)$ chosen points, considered as an effective divisor. From this last equality we deduce that $l(mH-\Lambda)\geq1$. Because $L(mH-\Lambda)$ is interpreted \cite{eisenbud1996} as the vector space of homogeneous polynomials of degree $m$ vanishing on $\Lambda$ modulo those of degree $m$ vanishing on $C$, we can ascertain that there exists a curve $f$ of degree $m$ passing through the $mn-\dfrac{1}{2}(n-1)(n-2)$ remaining points, and not vanishing identically on $C$.
We now operate on the equation \eqref{f6} by $f\left(\dfrac{\partial}{\partial x},\dfrac{\partial}{\partial y},\dfrac{\partial}{\partial z}\right)$. Then the terms corresponding to the points disappear on account of the relation
\begin{equation}
\label{f7}
\displaystyle f\left(\dfrac{\partial}{\partial x},\dfrac{\partial}{\partial y},\dfrac{\partial}{\partial z}\right)(\xi x+\eta y+\zeta z)^m=m!f(\xi,\eta,\zeta)
\end{equation}
and we are left with the equation
\begin{equation}
\label{f8}
\displaystyle\lambda_1f(\xi_1,\eta_1,\zeta_1)=0.
\end{equation}
Therefore either $f(\xi_1,\eta_1,\zeta_1)=0$, in which case all the points lie on a curve of the $m$-th degree and they would not have been chosen arbitrarily; or $\lambda_1=0$. But $\lambda_1$ can be taken to be anyone of the coefficients; hence all the coefficients are zero and the syzygy or linear relation \eqref{f6} does not exist.

This shows that we have $M=mn+1-\dfrac{1}{2}(n-1)(n-2)$ independent solutions $(s_j)_{1\leqslant j\leqslant M}$(or tangential equation) of the equation $F\left(\frac{\partial}{\partial x},\frac{\partial}{\partial y},\frac{\partial}{\partial z}\right) \phi(x,y,z)=0$; we obtain a linear relation between the $(m)^{th}$ powers of the tangential equations of any $mn+2-\dfrac{1}{2}(n-1)(n-2)$ points on the curve $C$ as follows: draw a curve $C^{m+1}$ (not containing $C$), $\chi(x,y,z)=0$, of degree $m+1$ through these points (this exists by a similar Riemann-Roch argument as before); this curve will cut the curve $C$ again in $\dfrac{1}{2}(n+1)(n-2)$ points by Bezout's theorem. Through these we may draw a curve $C^{n-2}$ of degree $n-2$, say $G(x,y,z)=0$. Now by the Cayley-Bacharach \theoref{CB1} together with the Serret's theorem, we may deduce that there is a linear relation between the $(m+n-2)^{th}$ powers of the tangential equations of the points of intersection of $F=0$ and $\chi=0$. This relation has the form
$$\sum_{r=1}^{n(m+1)}k_r(x\xi_r+y\eta_r+z\zeta_r)^{m+n-2}=0,\,k_r\in\C^\times.$$
Operating on this syzygy with $\displaystyle G\left(\dfrac{\partial}{\partial x},\dfrac{\partial}{\partial y},\dfrac{\partial}{\partial z}\right)$, the terms corresponding to $G=0$ disappear, and we are left with the equation
$$\sum_{1}^{mn+2-\frac{1}{2}(n-1)(n-2)}k_rG(\xi_r,\eta_r,\zeta_r)(x\zeta_r+y\eta_r+z\zeta_r)^{m}=0.$$

Similarly, when $m<n$ we may take $\dbinom{m+2}{m}$ points on the curve $C$ which do not lie on a curve of the $m^{th}$ degree. This is possible by \corref{Serret1} and the fact that $m<n$. 

We observe that for these so chosen $\dbinom{m+2}{m}$ the corresponding tangential equations are linearly independent. This follows from the fact that one can always draw a curve $C^m$ of degree $m$ through any $\dfrac{1}{2}m(m+3)$ given points by Riemann-Roch theorem, as before. Thus we obtain a linear relation between the $m^{th}$ powers of the tangential equations of any $\dfrac{1}{2}(m+2)(m+1)+1$ points on the curve, and which satisfy the equation \eqref{f3M}.

So the conclusion of what we have said so far is that we can find a basis of the space of homogeneous polynomials of degree $m$ which satisfy the partial differential equation:  $$F(\frac{\partial}{\partial x},\frac{\partial}{\partial y},\frac{\partial}{\partial z})\phi(x,y,z)=0$$ in the form $(\xi_i x+\eta_i y+\zeta_i z)^m,\,1\leqslant i\leqslant r$ for $r$ well-chosen points on $C$ with $r=mn+1-\dfrac{1}{2}(n-1)(n-2)$ in case $m\geqslant n$ and $r=\dbinom{m+2}{2}$ when $m<n$.

In the two cases if $r$ is the number of independent solutions, we have that
\begin{equation}
\label{f9}
\displaystyle(\xi(t) x+\eta(t) y+\zeta(t) z)^m=\sum_1^r(\xi_i x+\eta_i y+\zeta_i z)^m\mu_i(t)
\end{equation}
is a solution of the partial differential equation \eqref{f3M} with $\left[\xi,\eta,\zeta\right]\in C$, expressed as a function of the uniformizing parameter $t$. 

Indeed when $m<n$ we have $F\left(\dfrac{\partial}{\partial x},\dfrac{\partial}{\partial y},\dfrac{\partial}{\partial z}\right)(\xi x+\eta y+\zeta z)^m=0$ for degree reasons. And when $m\geqslant n$ we have

$$F\left(\dfrac{\partial}{\partial x},\dfrac{\partial}{\partial y},\dfrac{\partial}{\partial z}\right)(\xi_i x+\eta_i y+\zeta_i z)^m=F(\xi_i,\eta_i,\zeta_i)(\xi_i x+\eta_i y+\zeta_i z)^{m-n}.$$
The $\mu_i$ are analytic functions of $t$ because if $(\chi_i)_{1\leqslant i\leqslant r}$ is the basis dual to $\{(\xi_i x+\eta_i y+\zeta_i z)^m\}_{1\leqslant i\leqslant r}$ then $\chi_i((\xi (t)x+\eta(t) y+\zeta(t) z)^m)=\mu_i(t)$.

Let us take the $\mu_i$ as in \eqref{f9}; we remark that the $\mu_i$ are linearly independent over $\C$ as functions of $t$. Indeed if not, $(\xi(t) x+\eta(t) y+\zeta(t) z)^m$ would be expressible as a linear combination of $r-1$ solutions to equation \eqref{f3M}, and we would be able to obtain a syzygy among the tangential equations of any r points on the curve. More precisely if for instance $\mu_1$ is expressed as a linear combination of the other $(\mu_i)_{2\leqslant i\leqslant r}$ then we have a relation of the form
$$(\xi(t) x+\eta(t) y+\zeta(t) z)^m=\displaystyle\sum_1^{r-1}\alpha_j(t)H_j(x,y,z)$$
for fixed $(H_j(x,y,z))_{1\leqslant j\leqslant r-1}$ (solution to \eqref{f3M}). Let us take now $r$ arbitrary points $\left[\xi^0_k,\eta_k^0,\zeta_k^0\right]_{1\leqslant k\leqslant r}$ on the curve $C$ in the domain of the uniformizing parameter $t$. Then there exists $(t_k)_{1\leqslant k\leqslant r}$ so that $\left[\xi^0_k,\eta_k^0,\zeta_k^0\right]=\left[\xi(t_k),\eta(t_k),\zeta(t_k)\right]$, $1\leqslant k\leqslant r$. Thus the corresponding $m$-th powers of linear forms $(\xi_i^0 x+\eta_i^0 y+\zeta^0_k z)^m$, $1\leqslant k\leqslant r$ are expressed as linear combinations of the $r-1$ fixed given vectors $H_j$. Hence by linear algebra, $(\xi_i^0 x+\eta_i^0 y+\zeta^0_k z)^m$ are linear dependent. This contradicts the above.

We consider $r$ functions $(f_s)_{1\leqslant s\leqslant r}$ (analytic). We have

\begin{equation}
\label{f10}
\displaystyle \int(\xi x+\eta y+\zeta z)^mf_s(t)dt=\sum_1^r(\xi_i x+\eta_i y+\zeta_i z)^m\int\mu_i(t)f_s(t)dt.
\end{equation}
and the determinant of the matrix $(\displaystyle\theta_{s,i})_{1\leqslant s\leqslant r,1\leqslant i\leqslant r}:=(\int\mu_i(t)f_s(t)dt)_{1\leqslant i\leqslant r,1\leqslant s \leqslant r}$ will, in general, be non zero (this follows by a suitable generalization of \lemref{Stone} below). 

Accordingly we may choose $r$ constants $\lambda_j$, so that the expressions
\begin{equation}
\label{f11}
\displaystyle\lambda_1\theta_{1,i}+\lambda_2\theta_{2,i}+\ldots+\lambda_r\theta_{r,i}\qquad (i=1,2,\ldots,r)
\end{equation}
take any $r$ assigned values $p_1$,$\ldots$,$p_r$, and we have
\begin{equation}
\label{f12}
\displaystyle \int(\xi(t) x+\eta(t) y+\zeta(t) z)^m\sum_1^r(\lambda_sf_s)(t)dt=\sum_1^rp_i(\xi_i x+\eta_i y+\zeta_i z)^m.
\end{equation}
But any homogeneous polynomial of degree $m$ solution of equation \eqref{f3M} can be expressed in the form 
\begin{equation}
\label{f13}
\displaystyle \sum_1^rp_i(\xi_i x+\eta_i y+\zeta_i z)^m,
\end{equation}
and can therefore be put in the form
\begin{equation}
\label{f14}
\displaystyle \int(\xi x+\eta y+\zeta z)^mg_m(t)dt,\,\,g_m(t)=\sum_1^r(\lambda_sf_s)(t).
\end{equation}
The series $\displaystyle\sum_{m\geqslant0}((\xi x+\eta y+\zeta z)^mg_m)(t)$ converges uniformly on compact sets of the domain $\mathscr{D}$ of the uniformizing parameter $t$ provided that $|x|+|y|+|z|$ is small. Hence if we integrate on a compact path $\mathscr{L}$ of $\mathscr{D}$, we can write 
\begin{equation}
\label{f15}
\displaystyle V=\int\sum_{m\geqslant0}(\xi x+\eta y+\zeta z)^mg_m(t)dt=\int\Phi(\xi x+\eta y+\zeta z,t)dt
\end{equation}
which is the desired form of the solution.
\end{proof}
\begin{lemma}
\label{Stone}
Let $\nu_1$, $\ldots$, $\nu_N$ linearly independent over $\C$, continuous on a segment $\left[a,b\right]$, with $-\infty <a<b<+\infty$ to fix ideas. Then there exists $N$ continuous functions $l_1$, $\ldots$, $l_N$ with $$\det(\int_a^b\nu_i(u)l_j(u)du)\not=0.$$\end{lemma}
\begin{proof}
We define
$$F_k(s)=\int_a^b\nu_k(u)e^{su}du,s\in\C.$$
$F_k$ is entire, for every $k$. Besides $F_1$, $F_2$, $\ldots$, $F_N$ are linearly independent. Indeed if one has a relation
$$\displaystyle\sum_{k=1}^Nc_kF_k=0$$
then $\sum_{k=1}^Nc_k\int_a^b\nu_k(u)e^{su}du=0$ for all $s\in\C$. Therefore $\int_a^b(\sum_1^Nc_k\nu_k(u))u^pe^{su}du=0$ for all $s\in\C$ and $p\geqslant0$. So $\int_a^b(\sum_1^Nc_k\nu_k(u))R(u)du=0$ for every polynomial $R(u)$ by linearity and taking $s=0$. Hence by the Stone-Weiertrass theorem $\sum^N_1c_k\nu_k=0$ and $\lambda_k=0$ for every $k$ by linear independence of the $\nu_k$. Now the $(F_k)_{1\leqslant k\leqslant N}$ are linearly independent and they form a basis of the linear differential equation $W(y,F_1,\ldots, F_N)=0$ where $W$ is the Wronskian. And then the Wronskian determinant of $F_1$, $\ldots$, $F_N$ is non identically zero, hence it does not vanish for some $s_0\in\C$. So taking $l_j(u)=u^{j-1}e^{s_0u}$, $1\leqslant j\leqslant N$ we find that
$$\det(\int_a^b\nu_i(u)l_j(u)du)\not=0.$$
\end{proof}
\section{The kernel of the partial differential operator $F\left(\frac{\partial}{\partial x_1},\frac{\partial}{\partial x_2},\ldots,\frac{\partial}{\partial x_d}\right)$ }
\label{batemannew}
In this section we prove that one can always represent any real analytic function in the kernel of $F\left(\frac{\partial}{\partial x_1},\frac{\partial}{\partial x_2},\ldots,\frac{\partial}{\partial x_d}\right)$ on a sufficiently small open set, by means of a definite integral. Before doing so we need some preliminaries.

Let $X\subset \C P_{d}$ a projective variety with vanishing ideal $I(X)\subset \C[z_0 ,\ldots,z_d]$. For $m\geqslant1$, we denote by $I(X)_m$ the $m$-th homogeneous part $I(X)\cap \C[z_0,\ldots,z_d]_m$ of $I(X)$, where $\C[z_0,\ldots,z_d]_m$ is the subset of homogeneous polynomials of degree $m$ in $\C[z_0,\ldots,z_d]$. Since $I(X)$ is a homogeneous ideal, the homogeneous coordinate ring $S(X)$ is a graded ring with decomposition
$$S(X)=\displaystyle\bigoplus_{m\geqslant0}S(X)_m$$
where $S(X)_m=\C[z_0 ,\ldots,z_d]/I(X)_m$. Each homogeneous part $I(X)_m$ is a linear subspace of the $\dbinom{d+m}{d}$ dimensional $\C$-vector space $\C[z_0,\ldots,z_d]_m$. The dimension of $I(X)_m$ is the number of independent hypersurfaces of degree $m$ containing $X$. 

If $X\subset\C P_d$ is a hypersurface given by some irreducible homogeneous polynomial $F$ of degree $n$. The $m$-th homogeneous part $I(X)_m$ then consists of all polynomials of degree $m$ divisible by $F$. So we can identify $I(X)_m$ with $\C[z_0,\ldots,z_d]_{m-n}$ for $m\geqslant n$. So that 
$$\dim(I(X)_m)=\dbinom{m-n+d}{d}.$$

Let us introduce the following theorem of Serret \cite[p.~99-100]{semple1949}
\begin{theorem}[Serret]
\label{Serret2}
Consider $p=\left[a_0,a_1,\ldots,a_d\right]$ a point of $\C P_d$ and associate to it the linear form $l_p(z_0,z_1,\ldots,z_d)=a_0z_0+a_1z_1+\ldots a_nz_d$. Then the necessary and sufficient condition that any given hypersurface $C^r$, of given degree $r$, which passes through $q-1$ of a set of $q$ given points of $\C P_d$ should pass through the remaining point, is that there should be a linear relation (or syzygy) connecting the $r^{th}$ powers of the linear forms (or by abuse of language, tangential equation) associated to each given point.
\end{theorem}
\begin{proof}
The proof is similar to the case $d=2$. See \cite[p.~99-100]{semple1949}.
\end{proof}
\begin{corollary}
\label{Serret3}
When $q=\dbinom{m+d}{d}$ Serret's theorem gives the necessary and sufficient condition that $q$ points should lie on a hypersurface of degree $m$.
\end{corollary}
\begin{proof}
Indeed there always exists a hypersurface of degree $m$ which passes through any $q-1=\dbinom{m+d}{d}-1$ given points of the fixed smooth hypersurface $C$. So if there is linear relation between the $m^{th}$ powers of the linear terms associated to the q points, then necessarily all the q points lie on a hypersurface of degree $m$. The converse is the necessity statement of Serret's theorem.\end{proof}

\begin{theorem} 
\label{first12}
If $\left[\xi_1,\xi_2,\ldots,\xi_d\right]$ are the coordinates of a point on the smooth projective hypersurface $X$ of equation $F(\xi_1,\xi_2,\ldots,\xi_d)=0$ and of degree $n\geqslant2$, expressed as functions of uniformizing parameters $t_1,t_2,\ldots,t_{d-1}$, then any real analytic solution of the equation \begin{equation}
\label{f3M1}
\displaystyle F\left(\dfrac{\partial}{\partial x_1},\dfrac{\partial}{\partial x_2},\ldots,\dfrac{\partial}{\partial x_d}\right)\phi(x_1,x_2,\ldots,x_d)=0, \,d\geqslant4
\end{equation}
 on a sufficiently small open set can be put in the form
\begin{equation}
\label{f501}
\displaystyle V=\int \Phi(\xi_1 x_1+\xi_2 x_2+\ldots+\xi_d x_d, t_1,t_2,\ldots,t_{d-1})dt_1dt_2\ldots dt_{d-1},
\end{equation}
for a suitable region of integration.
\end{theorem}
\begin{proof}
To prove this we choose for origin a point in the vicinity of which $V$ is a real analytic function of $x_1$, $x_2,\ldots,$ $x_d$: we can expand $V$ as a power series in $x_1$, $x_2,\ldots,$ $x_d$ converging absolutely and uniformly within a certain region.

Operating on this series with $F\left(\dfrac{\partial}{\partial x_1},\dfrac{\partial}{\partial x_2},\ldots,\dfrac{\partial}{\partial x_d}\right)$ and equating to zero the coefficients of the various powers: we have when $m\geqslant n$ (we look at the homogeneous part of $F\left(\dfrac{\partial}{\partial x_1},\dfrac{\partial}{\partial x_2},\ldots,\dfrac{\partial}{\partial x_d}\right)V$ of degree $m$) 
$$\dbinom{m-n+(d-1)}{d-1}$$
relations among the coefficients of the homogeneous parts of degree $m$ as the action of  $F\left(\dfrac{\partial}{\partial x_1},\dfrac{\partial}{\partial x_2},\ldots,\dfrac{\partial}{\partial x_d}\right)$ on $V(x_1,x_2,\ldots,x_d)$ will leave a polynomial of degree $m-n$ and $\dbinom{m-n+(d-1)}{d-1}$ represents the dimension of the space of homogeneous polynomials of degree $m-n$ in $d$ variables. But when $m<n$ no such relations, because the operation of $F$ on the homogeneous parts of $V(x_1,x_2,\ldots,x_d)$ of degree $m\less n$ will kill them entirely.

As the dimension of the space of homogeneous polynomials of degree $m$ in $d$ variables is $\dbinom{m+(d-1)}{(d-1)}$, the terms of order $m$ are a linear combination of at most
$$M_1=\dbinom{m+(d-1)}{(d-1)}-\dbinom{m-n+(d-1)}{d-1},$$
linearly independent elements when $m\geqslant n$; and of at most $\dbinom{m+(d-1)}{d-1}$ independent terms when $m<n$.

In order to express the terms of order $m$ in the form \eqref{f501} in the case $m\geqslant n$ we proceed as follows. Take $M_1$ arbitrary points on the hypersurface $F(\xi_1,\xi_2,\ldots,\xi_n)=0$ (belonging to the same domain $\mathscr{D}$ of the uniformizing parameters $t_1,\ldots,t_{d-1}$); then the corresponding quantities $(\xi_1 x_1+\xi_2 x_2+\ldots+\xi_d x_d)^m$ will in general be linearly independent.

For, if not, there would be a linear relation between them; let it be
\begin{equation}
\label{f601}
\displaystyle\sum_1^{M_1}\lambda_i(\xi_1^i x_1+\xi_2^i x_2+\ldots+\xi_d^i x_d)^m=0.
\end{equation}
 Leaving out one of the points, say $(\xi_1^1,\xi_2^1,\ldots,\xi_n^1)$, we can draw a hypersurface of the $m$-th order $f(\xi_1,\xi_2,\ldots,\xi_n)=0$ through the remainder and not vanishing identically on $X$. Indeed with the notations introduced above we have
$$\dim(I(X)_m)=\dbinom{m-n+d-1}{d-1}$$
and this number is positive for $d\geqslant4$ and $m\geqslant n$. This quantity is the dimension of the space of hypersurfaces of degree $m$ containing $X$. Since we are considering $M_1-1$ points of $X$, the codimension of the space of forms of degree $m$ in the variables $z_1,z_2,\ldots,z_d$ in the space of forms of degree $m$ in those variables and vanishing on $X$ is at most $M_1-1$. So the dimension of the space of forms of degree $m$ in the $d$ variables $z_1,z_2,\ldots,z_d$ and passing through the $M_1-1$ points is at least $\dim\C[z_1,z_2,\ldots,z_d]_m-M_1+1=\dbinom{m-n+(d-1)}{d-1}+1$. So there is at least one non-zero homogeneous polynomial $f$, of degree $m$ in the $d$ variables vanishing on the chosen $M_1-1$ points but not vanishing entirely on $X$. 

We now operate on the equation \eqref{f601} with $f\left(\dfrac{\partial}{\partial x_1},\dfrac{\partial}{\partial x_2},\ldots,\dfrac{\partial}{\partial x_d}\right)$. Then the terms corresponding to the points disappear on account of the relation
\begin{equation}
\label{f701}
\displaystyle f\left(\dfrac{\partial}{\partial x_1},\dfrac{\partial}{\partial x_2},\ldots,\dfrac{\partial}{\partial x_d}\right)(\xi_1 x_1+\xi_2 x_2+\ldots+\xi_d x_d)^m=m!f(\xi_1 ,\xi_2 ,\ldots,\xi_d )
\end{equation}
and we are left with the equation
\begin{equation}
\label{f801}
\displaystyle\lambda_1f(\xi_1^1,\xi_2^1,\ldots,\xi_d^1)=0.
\end{equation}
Therefore either $f(\xi_1^1,\xi_2^1,\ldots,\xi_d^1)=0$, in which case all the points lie on a hypersurface of the $m$-th degree and they would not have been chosen arbitrarily; or $\lambda_1=0$. But $\lambda_1$ can be taken to be anyone of the coefficients; hence all the coefficients are zero and the syzygy or linear relation \eqref{f601} does not exist.

Thus we have $M_1$ independent solutions (or tangential equation) of the equation $$F\left(\dfrac{\partial}{\partial x_1},\dfrac{\partial}{\partial x_2},\ldots,\dfrac{\partial}{\partial x_d}\right)\phi(x_1,x_2,\ldots,x_d)=0.$$ 

In other words, there exists a linear relation between the $m$-th powers of the tangential equations of any $M_1+1$ points on the hypersurface, unless all but one of them lie on a hypersurface of the $m$-th degree, in which case the syzygy does not contain a point corresponding to the last point. 

Similarly, when $m<n$ we may take $\dbinom{m+(d-1)}{d-1}$ points on the hypersurface which do not lie on a hypersurface of the $m^{th}$ degree. This is possible by \corref{Serret3} and the fact that $m<n$. 

Having so chosen the $\dbinom{m+(d-1)}{d-1}$ points on $X$ we observe that the corresponding tangential equations are linearly independent. This follows from the fact that one can always draw a hypersurface of degree $m$ through any $\dbinom{m+(d-1)}{d-1}-1$ given points on $X$. 

Again we have the corresponding assertion that when $m<n$ there is a linear relation between the $m^{th}$ powers of the tangential equations of any $\dbinom{m+(d-1)}{d-1}+1$ points on the curve, and which satisfy the equation \eqref{f3M1}. 

So the conclusion of what we have said so far is that we can find a basis of the space of homogeneous polynomials which satisfy \eqref{f3M1} in the form $(\xi_1^i x_1+\xi_2^i x_2+\ldots+\xi_d^i x_d)^m,\,1\leqslant i\leqslant r$ for $r$ well-chosen points on $X$ with $r=M_1$ in case $m\geqslant n$ and $r=\dbinom{m+(d-1)}{d-1}$ when $m<n$.

Let us take the two cases together, and denote by $r$ the number of independent solutions, we see that
\begin{equation}
\label{f901}
\begin{split}
\displaystyle&(\xi_1(t_1,t_2,\ldots,t_{d-1}) x_1+\xi_2(t_1,t_2,\ldots,t_{d-1}) x_2+\ldots+\xi_d(t_1,t_2,\ldots,t_{d-1})x_d)^m\\&=\sum_1^r(\xi_1^i x_1+\xi_2^i x_2+\ldots+\xi_d^i x_d)^m\mu_i(t_1,t_2,\ldots,t_{d-1})
\end{split}
\end{equation}
is a solution of the pde \eqref{f3M1} with $\left[\xi_1,\xi_2,\ldots,\xi_d\right]\in X$, expressed as a function of the uniformizing parameters $t_1,t_2,\ldots,t_{d-1}$. 

Indeed when $m<n$ then $F\left(\dfrac{\partial}{\partial x_1},\dfrac{\partial}{\partial x_2},\ldots,\dfrac{\partial}{\partial x_d}\right)(\xi_1 x_1+\xi_2 x_2+\ldots+\xi_d x_d)^m=0$ for degree reasons. And in case $m\geqslant n$ we have

\begin{equation}\begin{split}&F\left(\dfrac{\partial}{\partial x_1},\dfrac{\partial}{\partial x_2},\ldots,\dfrac{\partial}{\partial x_d}\right)(\xi_1^i x_1+\xi_2^i x_2+\ldots+\xi_d^i x_d)^m\\&=F(\xi_1^i,\xi_2^i ,\ldots,\xi_d^i)(\xi_1^i x_1+\xi_2^i x_2+\ldots+\xi_d^i x_d)^{m-n}.\end{split}\end{equation}
The $\mu_i$ are analytic functions of $(t_1,t_2,\ldots,t_{d-1})$ because if $(\chi_i)_{1\leqslant i\leqslant r}$ is the basis dual to $\{(\xi_1^i x_1+\xi_2^i x_2+\ldots+\xi_d^i x_d)\}_{1\leqslant i\leqslant r}$ then $\chi_i((\xi_1(t_1,t_2,\ldots,t_{d-1}) x_1+\xi_2(t_1,t_2,\ldots,t_{d-1}) x_2+\ldots+\xi_d(t_1,t_2,\ldots,t_{d-1})x_d)^m)=\mu_i(t_1,t_2,\ldots,t_{d-1})$.

Let us take the $\mu_i$ as in \eqref{f901}; we remark that the $\mu_i$ are linearly independent over $\C$ as functions of $(t_1,t_2,\ldots,t_{d-1})$ by a reasoning similar to the one given in the proof of \theoref{second}.

We consider $r$ functions $(f_s)_{1\leqslant s\leqslant r}$ of $d$ variables $(t_1,t_2,\ldots,t_{d-1})$. We have

\begin{equation}
\label{f1001}
\begin{split}
\displaystyle &\int(\xi_1 x_1+\xi_2 x_2+\ldots+\xi_d x_d)^mf_s(t_1,t_2,\ldots t_{d-1})dt_1dt_2\ldots dt_{d-1}\\&=\sum_1^r(\xi_1^i x_1+\xi_2^i x_2+\ldots+\xi_d^i x_d)^{m}\int\mu_if_s(t_1,t_2,\ldots t_{d-1})dt_1dt_2\ldots dt_{d-1}
\end{split}
\end{equation}
and the vectors $(\displaystyle\theta_{s,i})_{1\leqslant i\leqslant r}:=(\int\mu_if_s(t_1,t_2,\ldots t_{d-1})dt_1dt_2\ldots dt_{d-1})_{1\leqslant i\leqslant r}$, where we have $s=1,\ldots,r$, will in general, be linearly independent (by a several variable lemma analogue to \lemref{Stone}).

Accordingly we may choose $r$ constants $\lambda_j$, so that the expressions
\begin{equation}
\label{f1101}
\displaystyle\lambda_1\theta_{1,i}+\lambda_2\theta_{2,i}+\ldots+\lambda_r\theta_{r,i}\qquad (i=1,2,\ldots,r)
\end{equation}
take any $r$ assigned values $p_1$,$\ldots$,$p_r$, and we have
\begin{equation}
\label{f1201}
\begin{split}
\displaystyle&\int (\xi_1 x_1+\xi_2 x_2+\ldots+\xi_d x_d)^m\sum_1^r\lambda_sf_s(t_1,t_2,\ldots,t_{d-1})dt_1dt_2\ldots dt_{d-1}\\&=\sum_1^rp_i(\xi_1^ ix_1+\xi_2^i x_2+\ldots+\xi_d^i x_d)^m.
\end{split}
\end{equation}
But any homogeneous polynomial of the degree $m$ solution of \eqref{f3M1} can be expressed in the form 
\begin{equation}
\label{f1301}
\displaystyle \sum_1^rp_i(\xi_1^ ix_1+\xi_2^i x_2+\ldots+\xi_d^i x_d)^m,
\end{equation}
and therefore admits the representation
\begin{equation}
\label{f1401}
\displaystyle \int(\xi_1 x_1+\xi_2 x_2+\ldots+\xi_d x_d)^mg_m(t_1,t_2,\ldots,t_{d-1})dt_1dt_2\ldots dt_{d-1}.
\end{equation}
The series $$\displaystyle\sum_{m\geqslant0}((\xi_1x_1+\xi_2(t_1,t_2,\ldots,t_{d-1}) x_2+\ldots+\xi_d x_d)^mg_m)(t_1,t_2,\ldots,t_{d-1})$$ converges uniformly on compact sets of the domain $\mathscr{D}$ of the uniformizing parameters $t_1$, $t_2$, $\ldots,t_{d-1}$ provided that $|x_1|+|x_2|+\ldots+|x_d|$ is small. Hence if we integrate on a compact submanifold $\mathscr{S}$ of $\mathscr{D}$, we can write 
\begin{equation}
\label{f1500}
\begin{split}
\displaystyle V&=\int\sum_{m\geqslant0}((\xi_1 x_1+\ldots+\xi_d x_d)^mg_m(t_1,\ldots,t_{d-1})dt_1\ldots dt_{d-1}\\
&=\int\Phi(\xi_1 x_1+\xi_2 x_2+\ldots+\xi_d x_d, t_1,t_2,\ldots,t_{d-1})dt_1dt_2\ldots dt_{d-1}.
\end{split}
\end{equation}
which is the desired form of the solution.

This being done for all values of $m$, our series for $V$ takes the required form
\begin{equation}
\label{f1501}
\displaystyle V(x_1,\ldots,x_d)=\int \Phi(\xi_1 x_1+\xi_2 x_2+\ldots+\xi_d x_d, t_1,t_2,\ldots,t_{d-1})dt_1dt_2\ldots dt_{d-1}.
\end{equation}
\end{proof}
\begin{example}
We treat the example of the partial differential equation 
$$\displaystyle\dfrac{\partial^{n-1}}{\partial x_1\partial x_2\ldots \partial x_{n-1}}\phi(x_1,x_2,\ldots,x_{n-1})=\dfrac{\partial^{n-1}}{\partial x_n^{n-1}}\phi(x_1,x_2,\ldots,x_n).$$
Using the following embedding of $(\C^{\times})^{n-2}$ into the associated characteristic variety of the equation: $z_1z_2\ldots z_{n-1}=z_n^{n-1}$
$$(t_1,t_2,\ldots,t_{n-2})\mapsto (t_1,t_2,\ldots,t_{n-2},\frac{1}{t_1t_2\ldots t_{n-2}},1)$$
we obtain an integral representation
$$\displaystyle\int_\mathscr{S}\Phi(x_1t_1,x_2t_2,\ldots x_{n-3}t_{n-2},\frac{x_{n-2}}{t_1t_2\ldots t_{n-2}},x_n,t_1,t_2,\ldots,t_{n-2})dt_{1}dt_{2}\ldots dt_{n-2}$$
for a suitable region $\mathscr{S}$ in $(\C^{\times})^{n-2}$.

We consider now the example of the partial differential equation
$$(\dfrac{\partial^3}{\partial x^3}+\dfrac{\partial^3}{\partial y^3}+\dfrac{\partial^3}{\partial z^3})\phi(x,y,z)=0.$$
We explain how a parametrization of the Fermat cubic $x^3+y^3=1$ arises, following Dixon \cite{dixon1890}. A natural idea to parametrize this curve is to find two meromorphic (eventually multivalued) functions $c(u)$, $s(u)$ such that
$$c^3(u)+s^3(u)=1.$$
We require, following \cite{dixon1890} the nonlinear differential system
\begin{equation}
\begin{split}
&s^\prime=c^2,\,\,c^\prime=-s^2\\
&s(0)=0,\,\,c(0)=1.
\end{split}
\end{equation}
These functions are analytic about the origin from the theory of ordinary differential equations. Besides 
$$s^3(u)+c^3(u)=1$$
because
$$(s^3+c^3)^\prime=3s^2c^2-3c^2s^2=0$$
and by making use of the initial conditions. So $(s(u),c(u))$ gives a parametrization of the curve $x^3+y^3=1$ near the point $(0,1)$. Dixon established that the functions are meromorphic in the whole of the complex plane and doubly periodic (that is, elliptic), hence they provide a global parametrization of the Fermat cubic. More precisely one can solve the differential system for $s(u)$ and $c(u)$ as usual by elimination
$$s^\prime=c^2\Longrightarrow s^{\prime\prime}=2cc^\prime=-2cs^2=-2s^2\sqrt{s^\prime}.$$
This gives
$$s^{\prime\prime}\sqrt{s^\prime}=-2s^2s^\prime.$$
Thus after one integration we have
$$\dfrac{2}{3}(s^\prime)^{3/2}=-\dfrac{2}{3}s^3+K.$$
Since $s(0)=0$, $s^\prime(0)=1$ we get $K=2/3$ and
$$(s^{\prime})^{3/2}=-s^3+1.$$
This gives 
$$\dfrac{ds}{du}=(1-s^3)^{2/3},$$
$$\dfrac{du}{ds}=(1-s^3)^{-2/3}\Longrightarrow du=(1-s^3)^{-2/3}ds.$$
So
$$\displaystyle u=\int_0^udz=\displaystyle\int_0^{s(u)}\dfrac{dt}{(1-t^3)^{2/3}}.$$
Similarly we obtain
$$u=\int_0^udz=\displaystyle\int_{c(u)}^1\dfrac{dt}{(1-t^3)^{2/3}}.$$
Expanding $(1-t^3)^{2/3}$ near $0$ (respectively near $1$) and using Lagrange inversion we arrive at the formulas given by
\begin{equation}
\begin{split}
&s(u)=u-4\frac{u^4}{4!}+160\frac{u^7}{7!}-20800\frac{u^{10}}{10!}+6476800\frac{u^{13}}{13!}-\ldots\\
&c(u)=1-2\frac{u^3}{3!}+40\frac{u^6}{6!}-3680\frac{u^{9}}{9!}+8880000\frac{u^{12}}{12!}-\ldots.
\end{split}
\end{equation}
Using this we have an integral representation for the solution
$$V(x,y,z)=\displaystyle\int_\gamma\Phi(c(u)x+s(u)y-z,t)dt$$
for a suitable path $\gamma$.
\end{example}

\section{Appendix 1: The twistor correspondence and solutions of differential equations}
\label{twistor}
In this appendix we present an interesting approach to the determination of the solutions of linear partial differential equations due to \cite{murray1985}.
Let $f(z_0,z_1,\cdots,z_n)$ be a complex homogeneous polynomial of degree $k>1$ in the indicated variables with $n>1$. In this section we will show how to use a general twistor correspondence to describe all the solutions of
\begin{equation}
\label{p2}
\displaystyle D_f\phi:=f\left(\dfrac{\partial}{\partial z_0},\ldots,\dfrac{\partial}{\partial z_n}\right)\phi(z_0,\ldots,z_n)=0.
\end{equation} 
We will do so by showing that there is a twistor space $Z$, a vector space $H^{n-1}(Z,\mathcal{O}(-n-1+k))$ of Dolbeault cohomology classes and a twistor correspondence
\begin{equation}
\label{p3}
T:H^{n-1}(Z,\mathcal{O}(-n-1+k))\to H^0(\C^{n+1},\mathcal{O}),
\end{equation} 
whose image is the space of solutions of \ref{p2}.
Here and in what follows we will denote by $\mathcal{O}$ is the sheaf of holomorphic functions on $\C^m$ for some $m\in \N^\times$, and we will show that $T$ is an injective map, obtained by integrating the cohomology class against "a cycle", and onto the space of analytic functions in the kernel of $D_f$ for all $k>1$. 

To begin with, we recall that the complex projective space $\C P_n$ is the set of all lines through the origin in $\C^{n+1}$, namely the set of all $\left[z\right]$, $z\in\C^{n+1}$-$\{0\}$ with
$$\displaystyle\left[z\right]:=\left[z_0:z_1:\ldots:z_n\right]=\{\lambda z,\lambda\in\C^{\times}\}.$$
It is a compact complex manifold with a covering family of charts given by the open sets $U_i$ where the $i$-th homogeneous coordinate is non-zero.
\begin{equation*}
\label{p4}
\displaystyle\begin{split}
U_i&\to \C^n\\
\left[z_0:z_1:\ldots:z_n\right]&\mapsto\left(\dfrac{z_0}{z_i},\ldots,\dfrac{z_{i-1}}{z_i},\dfrac{z_{i+1}}{z_i},\ldots,\dfrac{z_n}{z_i}\right).
\end{split}
\end{equation*}
As in the previous sections let $H$ be the tautological line bundle over $\C P_n$ whose fibre over a point $\left[z\right]\in \C P_n$ is just the line $\left[z\right]$, i.e.
\begin{equation}
\label{p5}
\displaystyle H=\{(\left[z\right],w),w=\lambda z,\lambda\in \C^\times\}\subset\C P_n\times \C^{n+1}.
\end{equation}
We define local sections $\psi_i:U_i\to H$ for each $i=0,\ldots,n$ by
\begin{equation}
\label{p6}
\displaystyle\begin{split}
&U_i\to H\\
&\psi_i(\left[z\right])=\left(\left[z\right],\left(\dfrac{z_0}{z_i},\ldots,\dfrac{z_{i-1}}{z_i},1,\dfrac{z_{i+1}}{z_i},\ldots,\dfrac{z_n}{z_i}\right)\right).
\end{split}
\end{equation}
Note that $\psi_i$ is $H$-valued because
$$\left(\dfrac{z_0}{z_i},\ldots,\dfrac{z_{i-1}}{z_i},1,\dfrac{z_{i+1}}{z_i},\ldots,\dfrac{z_n}{z_i}\right)=\dfrac{1}{z_i}\left(z_0:z_1:\ldots:z_n\right).$$
From the definition of the $\psi_i$ we have that
$$\psi_i(\left[z\right])=\dfrac{z_j}{z_i}\psi_j(\left[z\right])$$
and hence the transition functions of $H$ are 
$$g_{ij}=\dfrac{z_j}{z_i}.$$
We set $\mathcal{O}(-1):=H$. For $p>0$ we define $\mathcal{O}\left(p\right)=\displaystyle\bigotimes_1^pH^\star$, with $H^\star$ the line bundle dual to $H$. We also define $\displaystyle\mathcal{O}\left(p\right)=\bigotimes_1^{-p}H$, for $p<0$ and $\mathcal{O}(0)$ as the trivial line bundle on $\C P_n$. 

For a given sheaf $\mathcal{S}$ on $\C P_n$, we denote by $H^p(\C P_n,\mathcal{S})$ its $p$-th cohomology group $p\geqslant 0$. The dimension of the vector space $H^p(\C P_n,\mathcal{S})$ (which is finite by Hodge theory) is denoted by 
$h^p(\C P_n,\mathcal{S})$. We have the following formulas of Bott \cite[p. 4]{okonek2010}
\begin{equation}
\label{p1015}
h^q(\C P_n,\mathscr{O}(k))=\begin{cases}
\dbinom{k+n}{k},\,k\geqslant0,\,q=0\\
\dbinom{-k-1}{-k-1-n},\,q=n,\,k\leqslant-n-1\\
0\,\,\,\text{otherwise}.
\end{cases}
\end{equation}

We have \cite[p. 165]{griffiths1978}
$$Sym^d((\C^{n+1})^\star)\simeq H^0(\C P_n,\mathscr{O}(d))$$
here $Sym^d((\C^{n+1})^\star)$ is the space of homogeneous polynomials of degree $d$ in the variables $z_0, \ldots,z_n$. Moreover from \cite[p. 135]{griffiths1978} one knows that to any global section in $H^0(\C P_n,\mathscr{O}(d))$, one can associate an effective divisor which is precisely where the considered section vanishes.
\begin{definition}
A divisor $D$ is a locally finite formal linear combination 
$$D=\sum a_i V_i$$
of irreducible analytic varieties $V_i$. If all the $a_i\geqslant0$, $D$ is called effective.
\end{definition}
Let us now go back to the homogeneous polynomial $f(z_0,z_1,\ldots,z_n)$ of degree $k$ we started with.  From what we have explained we can associate to it a section $f(\xi_0,\ldots,\xi_n)$ of $\mathscr{O}(k)$ which vanishes exactly on an effective divisor denoted $X$.

We then have the exact sequence of sheaves
\begin{equation}
\label{p1016}
\begin{CD}
	\displaystyle0 @>>> \mathscr{O}(-k)@>>> \mathscr{O}_{\C P_n} @>>> \mathscr{O}_X @>>> 0.
	 \end{CD}\end{equation} 
By taking the tensor product of the sequence \eqref{p1016} with the locally free sheaf $\mathscr{O}(p)$ and using the Bott's formulas we get (by the long exact sequence in cohomology) the short exact sequences

\begin{equation}
\label{p1017}
\begin{CD}
	\displaystyle0 @>>> H^0(\C P_n,\mathscr{O}(p-k))@>>> H^0(\C P_n,\mathscr{O}(p)) @>>> H^0(X,\mathscr{O}(p)) @>>> 0.
	 \end{CD}\end{equation} 
and
\begin{equation}
\label{p1018}
\begin{CD}
	\displaystyle0 @>>> H^{n-1}(\C P_n,\mathscr{O}(p))@>>> H^n(\C P_n,\mathscr{O}(p-k)) @>>> H^n(\C P_n,\mathscr{O}(p)) @>>> 0
\end{CD}\end{equation} 

Set $p=-n-1=k$ in \eqref{p1018}. Then since $k>0$ we have $-n-1+k>-n-1$ and from formulas \eqref{p1015} we get
$$H^{n-1}(\C P_n,\mathscr{O}(-n-1+k))\simeq H^{n}(\C P_n,\mathscr{O} (-n-1))\simeq\C$$
where the last isomorphism also follows from \eqref{p1015}.

If we assume $X$ to be smooth, this isomorphism can be realized by integrating smooth $(0,n-1)$-forms with values in $\mathscr{O}(-n-1+k)$. 
 \begin{equation}
 \label{p1019}
 \begin{CD}
	\displaystyle0 @>>> \mathscr{O}(-n-1+k)@>>> \xi^{0,\bullet} \otimes\mathscr{O}(-n-1+k).	 
\end{CD}
\end{equation}
Recall now that $H^0(\C P_n,\mathcal{O}(1))$ is the space of homogeneous polynomials of degree $1$ in $\C^{n+1}$, and let $\xi_0,\ldots,\xi_n$ be a basis of $H^0(\C P_n,\mathcal{O}(1))$. We then have an isomorphism
\begin{equation}
\label{p7}
\displaystyle
\begin{split}
\C^{n+1}&\to H^0(\C P_n,\mathcal{O}(1))\\
(x_0,\ldots,x_n)&\mapsto\sum_{i=0}^{n}x_i\xi_i.
\end{split}
\end{equation} 
Moreover the restriction map
\begin{equation}
\label{p8}
H^0(\C P_n,\mathcal{O}(1))\to H^0(X,\mathcal{O}(1))
\end{equation}
is an isomorphism by Lefschetz's hyperplane theorem, and therefore if $X$ is the divisor defined by $f$ we obtain
\begin{equation}
\label{p9}
\C^{n+1}=H^0(\C P_n,\mathcal{O}(1))=H^0(X,\mathcal{O}(1)).
\end{equation}
\begin{remark}
If $E$ is a vector bundle over a manifold $M$ then there is a sequence of induced jet bundles $J^pE$, $p=0,1,\ldots$. The fibre of $J^pE$ at $m\in M$ is the vector space of all $p$-jets at $m$ or of equivalence classes of local sections about $m$ where two sections are equivalent if they have the same Taylor series up to order $p$. This equivalence relation is independent of the charts necessary to define it and the fibers fit together to form smooth vector bundles. There is a natural projection from $J^pE$ to $J^{p-1}E$ which "forgets" the $p$-th term in the Taylor series. Associated with the jet bundles one has the following exact sequence
\begin{equation}
\label{p10}
\begin{CD}
	\displaystyle0 @>>> S^pT^\star M\bigotimes E@>>> J^pE @>>> J^{p-1}E @>>> 0.
\end{CD}\end{equation} 
In this invariant language a $p^{th}$ order differential operator acting on sections of $E$ is just a bundle map
\begin{equation}
\label{p11}
D:J^pE\to E.
\end{equation}
When $E=\C^{n+1}\times \C$, $J^pE$ will be denoted $J^p\C$ and \eqref{p10} becomes
\begin{equation}
\label{p12}
\begin{CD}
	\displaystyle0 @>>> (\C^{n+1}\times S^p\C^{n+1})@>>> J^p\C @>>> J^{p-1}\C @>>> 0,
\end{CD}\end{equation} 
where $\C^{n+1}$ is identified with its dual. If $p=k$, then the natural flat connection on $\C^{n+1}$, which is just the exterior derivative, defines a splitting $J^k\C\to(\C^{n+1}\times S^k\C^{n+1})$ and composing with the projection on $\C^{n+1}$ followed by the composition with $f$
 defines the homogeneous polynomial differential operator 
 \begin{equation}
 \label{p13}
 D_f:J^k\C\to \C.
 \end{equation}
 This is how one should interpret the differential operator given in equation \eqref{p2}.
 \end{remark}
 Now let us define the twistor space $Z$ as the total s(ace of the line bundle $\mathcal{O}(1)$ on $\C P_n$ restricted (or pulled-back) to $X$. 
 
 The relation between $Z$ and $\C^{n+1}$ revolves around the following double fibration where we identify $\C^{n+1}$ with $H^0(X,\mathcal{O}(1))$ via the isomorphism \eqref{p9}.
\begin{equation}
\label{p14}
\vcenter{
\xymatrix{\displaystyle
& \C^{n+1}\times X\ar[ld]_\mu \ar[rd]^\nu \\
\C^{n+1}  & & Z }
}
\end{equation}
If $\displaystyle(\sum_0^nx_i\xi_i,z)\in\C^{n+1}\times X$, then the left hand map sends it to $(x_1,\ldots,x_n)$ and the right hand arrow sends it to $\displaystyle\sum_0^n\xi_i(z)$. If we fix a point $x\in\C^{n+1}$ then the image of the section
\begin{equation}
\label{p15}
\displaystyle \sum_0^nx_i\xi_i:X\to Z
\end{equation}
is a subvariety $X_x$ of $Z$ which is identified with $X$ by the projection map $\pi:Z\to X$. We will need various line bundles on $Z$. We use the mappping $\pi$ to pull-back line bundles from $X$. That is if $L$ is a line bundle on $X$, we define a line bundle $\pi^\star L$ on $Z$ by $(\pi^\star L)_z=L_{\pi(z)}$ (the fiber of $(\pi^\star L)$ over $z$ is by definition given by $L_{\pi(z)}$). Notice that a peculiar thing happens for the bundle $\mathcal{O}(1)$. Here if $z$ is an element of $Z$ then because $Z$ is itself the line bundle $\mathcal{O}(1)$, $z\in\mathcal{O}(1)_{\pi(z)}$, that is $z\in(\pi^\star\mathcal{O}(1))_z$. We denote this section of $\mathcal{O}(1)$, the one that sends $z$ to $z$, by $\eta$. It is the section whose divisor is the zero section $Z_0$ of the line bundle $Z\overset{\pi}{\to}X$. Notice that the subvariety $X_x$ in this notation is the subset of $Z$ where
\begin{equation}
\label{p16}
\displaystyle\eta=\sum_0^n\xi_ix_i.
\end{equation} 
Let $\omega_l,\, l=0,\ldots,m$ be $(0,n-1)$ forms on $X$ with values in $\mathcal{O}(-n-1+k-l)$. Consider now the $(0,n-1)$ form on $Z$ with values in $\mathcal{O}(-n-1+k)$ defined by
\begin{equation}
\label{p17}
\displaystyle\omega=\sum_0^m\pi^\star(\omega_l)\eta^l.
\end{equation}
Because each $X_x$ is a copy of $X$ we can restrict $\omega$ to $X_x$ and integrate. We obtain
\begin{equation}
\label{p18}
\displaystyle
\begin{split}
\int_{X_x=\pi^{-1}(X)}\omega&=\int_{X_x}\sum_0^m\pi^\star(\omega_l)\eta^l\\
&=\int_{X_x}\sum_0^m\pi^\star(\omega_l)(\sum_0^n\xi_ix_i)^l\\
&=\int_X\sum_0^m\omega_l(\sum_0^n\xi_ix_i)^l=:\phi(x).
\end{split}
\end{equation}
This $\phi$ is clearly in $\ker D_f$, because if we apply a monomial differential operator 
$$\left(\dfrac{\partial}{\partial x_0}\right)^{i_0}\left(\dfrac{\partial}{\partial x_1^{i_1}}\right)^{i_1}\ldots\left(\dfrac{\partial}{\partial x_n}\right)^{i_n}$$
to $\phi(x)$ we obtain
$$\displaystyle\int_X\sum_{l=k}^m\omega_ll(l-1)\ldots(l-k+1)(\sum_0^n\xi_ix_i)^{l-k}\xi^{i_1}\xi^{i_2}\ldots\xi^{i_k}.$$
Hence if we apply $D_f$ to $\phi(x)$ we obtain
$$\displaystyle\displaystyle\int_X\sum_0^m\omega_ll(l-1)\ldots(l-k+1)(\sum_0^n\xi_ix_i)^{l-k}f(\xi^{0},\xi^{1},\ldots,\xi^{k})$$
which vanishes because $f$ vanishes on $X$. 

In \cite{murray1985} it is shown that more generally one should integrate over $X_x$ elements belonging to $H^{n-1}(Z,\pi^\star\mathcal{O}(-n-1+k))$. 

Finally, as we indicated at the beginning of this section, one has in full generality the twistor transform 
\begin{equation}
\label{p19}
\begin{split}
T:H^{n-1}(Z,\pi^\star\mathcal{O}(-n-1+k))&\to H^0(\C^{n+1},\mathcal{O})\\
T(\omega)(x)&=\int_{X_x}\omega,
\end{split}
\end{equation}
which is a bijection onto the kernel of $D_f $ for $k\leqslant n$.
\begin{remark}
$H^{n-1}(Z,\pi^\star\mathcal{O}(-n-1+k))$ is an infinite dimensional vector space because $Z$ is non-compact and $-n-1+k<0$ as $k\leqslant n$. Also the appearance of $\mathcal{O}(-n-1+k)$ is due to the fact that the canonical bundle of $X$ is exactly $\mathcal{O}(-n-1+k)$. Moreover the fact that one should integrate $(0,n-1)$-forms is a consequence of the Dolbeault resolution which implies that $H^{n-1}(X,\mathscr{O}(-n-1+k))\simeq H^{n-1}(\Gamma(X,\mathcal{E}^{0,\bullet}(\mathscr{O}(-n-1+k))))$. Finally if $f(x^0,x^1,\ldots,x^n)$ is a homogenous polynomial then $f(\xi)=f(\xi^0,\xi^1,\ldots,\xi^n)$ defines a section of the line bundle $\mathscr{O}(k)$ which vanishes precisely on $X$, namely $\{f=0\}$. 
\end{remark}
Let us now see from what we have said in this section how the classical contour integral formula, for the Laplacian in three dimensions given in \cite{whittaker1927}, arises. First, from \cite{whittaker1927}, the general solution of the Laplacian in three dimensions is
\begin{equation}
\label{zwelve}
\phi(x,y,z)=\displaystyle\int_{-\pi}^\pi f(z+ix\cos u+iy\sin u,u)du
\end{equation}
for an arbitrary real analytic function $f$. In three dimensions the manifold $X$ is the quadric $Q$ which is the vanishing set in $\C P_2$ of $z^2_0+z_1^2+z_2^2$. We have $\C P_2\supseteq Q\simeq\C P_1$. From equation \eqref{p18} the general solution is 
\begin{equation}
\phi(x,y,z)=\displaystyle\int_{Q} \sum_0^\infty f_p(\xi)(x\xi_0+y\xi_1+z\xi_2)^pd\xi.
\end{equation}
Indeed from \eqref{p18}, and with the notations used before, $\displaystyle \sum_{0}^m\omega_p(\sum_0^n\xi_ix_i)^p$ is a $(0,1)$ form with values in $\mathscr{O}(-1)$. A crucial remark is that for any complex projective hypersurface $M=\{g=0\}$, of degree $k$ of $\C P^n$,  its canonical bundle with sheaf of sections given by the local holomorphic one forms on M, is isomorphic to $\mathscr{O}(-n-1+k)$. Therefore in the case at hand $\mathscr{O}(-1)$ is the canonical bundle of $Q$, and $\displaystyle \sum_{0}^m\omega_p(\sum_0^n\xi_ix_i)^p$ is a $(1,1)$-form on $Q$ (where we interpret $d\xi$ as a local $(1,0)$-form on $Q$).

But $f(\xi_0,\xi_1,\xi_2)=0$ precisely on $Q$. Thus we can identify $(\xi_0,\xi_1,\xi_2)$ as a variable point which lies on $Q$. The embedding of $\C$ in the quadric which will be used is 
$$w\mapsto\left[i(w^2+1),(w^2-1),2w\right].$$
Restricting to the circle this gives
$$u\mapsto\left[i(e^{2iu}+1),(e^{2iu}-1),2e^{iu}\right],\,u\in\left[-\pi,\pi\right]$$
where square brackets denote the point of $\C P_2$ which is the line through $(i(e^{2iu}+1),(e^{2iu}-1),2e^{iu})$. So $(x\xi_0+y\xi_1+z\xi_2)$ becomes
$$2e^{iu}(xi\cos u+yi\sin u+z).$$
Therefore we have
\begin{equation}
\begin{split}
\phi(x,y,z)&=\displaystyle\displaystyle\int_{-\pi}^\pi\sum_pf_p(u)2^pe^{ipu}(z+ix\cos u+iy\sin u)^ph(u)du\\&=\int_{-\pi}^\pi f(z+ix\cos u+iy\sin u,u)du
\end{split}
\end{equation}
by change of variables and this recovers the result \eqref{zwelve}.

We now note that the formula in \cite{whittaker1927} for a solution of the wave equation is 
\begin{equation}
\label{dreizehen}
\phi(x,y,z,t)=\displaystyle\int_{-\pi}^\pi\int_{-\pi}^\pi f(t+x\sin u\cos v+y\sin u\sin v+z\cos u,u,v)dudv.
\end{equation}
Again the hypersurface $X\subseteq\C P_3$ is the quadric $z^2_0+z_1^2+z_2^2-z_3^2=0$ which is simply $\C P_1\times \C P_1,$ which accounts for the double coutour integral in \eqref{dreizehen} and $Z=\mathscr{O}(1)|_X$. To obtain \eqref{dreizehen} one uses the $2$-$1$ ramified map given by
$$\C\times\C\to Q\subseteq \C P_3,$$
$$(\xi,\xi^\prime)\mapsto\left[(1+\xi^2)(-1+\xi^{\prime2}),-i(-1+\xi^2)(-1+\xi^{\prime2}),2i(1+\xi^{2})\xi^\prime,4i\xi\xi^\prime\right].$$
Setting $\xi=e^{iu},\,\xi^\prime=e^{iv}$ (we restrict to the product of circles) and using formula \eqref{p18}, we obtain
\begin{equation}
\begin{split}
\phi(x,y,z,t)&=\displaystyle\displaystyle\int_{\xi}\int_{\xi^\prime}\sum_pf_p(\xi,\xi^\prime)(2i(1+\xi^{2})\xi^\prime z+x(1+\xi^2)(-1+\xi^{\prime2})\\
&-i(-1+\xi^2)(-1+\xi^{\prime2})y+4i\xi\xi^\prime t)^pd\xi d\xi^\prime\\&=\int_{-\pi}^\pi\int_{-\pi}^\pi\sum_pf_p(u,v)(t+x\sin u\cos v+y\sin u\sin v+z\cos u)^pdudv\\
&=\displaystyle\int_{-\pi}^\pi\int_{-\pi}^\pi f(t+x\sin u\cos v+y\sin u\sin v+z\cos u,u,v)dudv.
\end{split}
\end{equation}

\end{document}